\documentclass[10pt]{amsart}
\usepackage{amsmath, amsfonts, amscd, latexsym, amsthm, amssymb}

\usepackage{hyperref}

\def \c{\mathbb{C}}
\def \z{\mathbb{Z}}
\def \r{\mathbb{R}}
\def \n{\mathbb{N}}
\def \p{\mathbb{P}}

\def \K{\mathcal{K}}

\def \.{\cdot}
\def \Div{\textup{Div}}

\def \deg{\textup{deg}}
\def \dim{\textup{dim}}

\def \Pic{\textup{Pic}}

\def \Vol{\textup{Vol}}

\def \supp{\textup{supp}}

\def \Ord{\textup{Ord}}

\def \ord{\textup{ord}}

\def \K{{\bf K}_{rat}(X)}
\def \qp{quasi-projective }
\def \ratmap{\dashrightarrow}

\theoremstyle{plain}
\newtheorem{Th}{Theorem}[section]
\newtheorem{Lem}[Th]{Lemma}
\newtheorem{Prop}[Th]{Proposition}
\newtheorem{Cor}[Th]{Corollary}

\theoremstyle{definition}
\newtheorem{Ex}[Th]{Example}
\newtheorem{Def}[Th]{Definition}
\newtheorem{Rem}[Th]{Remark}

\begin{document}
\title{Mixed volume and an extension of intersection theory of divisors}
\author{Kiumars Kaveh, A. G. Khovanskii\\Department of Mathematics\\
University of Toronto}
\address{Kiumars Kaveh: Department of Mathematics, University of Toronto, Toronto,
Canada,}
\email{kaveh@math.toronto.edu}

\address{A. G. Khovanskii: Department of Mathematics, Toronto University, Toronto,
Canada; Moscow Independent Univarsity; Institute for Systems Analysis, Russian Academy of Sciences,}
\email{askold@math.toronto.edu}
\thanks{The second author is partially supported by Canadian Grant {\rm N~156833-02}.}
\maketitle

\begin{abstract}
Let $\K$ be the collection  of all non-zero finite
dimensional subspaces of rational functions on an $n$-dimensional irreducible
variety $X$. For any $n$-tuple  $L_1, \ldots, L_n\in \K $, we define
an intersection index $[L_1, \ldots, L_n]$
as the number of solutions in $X$ of a system of equations $f_1 =
\cdots = f_n = 0$ where each $f_i$ is a generic function from the
space $L_i$. In counting the solutions, we neglect the solutions $x$ at
which all the functions in some space $L_i$ vanish as well as the
solutions at which at least one function from some subspace $L_i$ has a
pole. The collection $\K$  is a commutative semigroup
with respect to a natural multiplication. The intersection index
$[L_1, \ldots, L_n]$ can be extended to the Grothendieck group
of $\K$. This gives an extension of the intersection
theory of divisors. The extended theory is applicable even to non-complete varieties.  We show that this
intersection index enjoys all the main properties of the mixed volume
of convex bodies. Our paper is inspired by the Bernstein-Ku\v{s}nirenko
theorem from the Newton polytope theory.
\end{abstract}

\tableofcontents

\noindent{\it Key words:} System of algebraic equations, mixed volume
of convex bodies, Bernstein-Ku\v{s}nirenko theorem, linear system on
a variety, Cartier divisor, intersection index.\\

\noindent{\it AMS subject classification:} 14C20, 52A39\\

\section{Introduction} \label{sec-intro}
The present paper is the first in a series of three papers in which
we extend the well-known Bernstein-Ku\v{s}nirenko theorem to a
far more general setting. It is a revised and expanded version of the first half
of the preprint \cite{Askold-Kiumars-arXiv}. In this introduction we
discuss the content of the paper as well as a brief discussion of
the two coming papers.

The Bernstein-Ku\v{s}nirenko theorem computes the number of
solutions in $(\c^*)^n$ of a system of equations $P_1 = \cdots = P_n
= 0$ where each $P_i$ is a generic function from a fixed finite
dimensional vector space of functions spanned by monomials. The
answer is given in terms of the mixed volumes of the Newton
polytopes of the equations in the system (Section \ref{subsec-B-K}).

In this paper, instead of $(\c^*)^n$ we take any irreducible $n$-dimensional
complex algebraic variety $X$ \footnote{Throughout the paper,
all the varieties are assumed to be over complex numbers, and a variety is not automatically
assumed to be irreducible.}, and instead of a space of
functions spanned by monomials, we consider any non-zero finite dimensional
vector space of rational functions on $X$. For any $n$-tuple of such
finite dimensional spaces $L_1, \ldots, L_n$ of rational functions
on $X$ we define an {\it intersection index} $[L_1, \ldots, L_n]$ as the
number of solutions in $X$ of a system of equations $f_1 = \cdots =
f_n = 0$, where each $f_i$ is a generic function from the space
$L_i$. In counting the solutions, we neglect the solutions $x$ at which
all the functions in some space $L_i$ vanish, as well as the solutions at
which at least one function from some space $L_i$ has a pole.
As this intersection index only depends on the birational type of $X$,
we will also refer to it as the {\it birationally invariant intersection index}.

Let $\K$ denote the collection of all the non-zero finite
dimensional subspaces of rational functions on $X$. The set $\K$ has a natural
multiplication: product $L_1L_2$ of two
subspaces $L_1, L_2 \in \K$ is the subspace spanned
by all the products $fg$ where $f \in L_1$, $g \in L_2$. With
this multiplication, $\K$ is a commutative
semigroup. As with any other commutative semigroup, there
corresponds a Grothendieck group to the semigroup $\K$ \footnote{For a commutative semigroup $K$, the
Grothendieck group $G(K)$ is defined as follows: for $x, y \in K$ let us say $x \sim y$ if there
is $z \in K$ with $xz=yz$. Then $G(K)$ consists of
all the formal quotients $x/y$, $x,y \in K$, where we identify $x/y$ and $z/w$ if
$xw = zy$. One has a natural homomorphism $\phi: K \to G(K)$ with the following
universal property: for any abelian group $G'$ and a homomorphism $\phi':K
\to G'$, there exists a unique homomorphism $\psi: G(K) \to G'$ such that
$\phi' = \psi \circ \phi$ (Section \ref{sec-Grothendieck-gp-Cartier}).}. We
prove that the intersection index $[L_1, \ldots, L_n]$ is linear
in each argument with respect to the multiplication in $\K$, and hence can be
extended to the Grothendieck group of $\K$.

With each space $L \in \K$, one associates a rational {\it Kodaira
map} $\Phi_L: X \ratmap \p(L^*)$, the projectivization of the dual space
$L^*$. Let $x \in X$ be such that all the $f \in L$ are defined at $x$ and not all of them
vanish at $x$. To such $x$ there corresponds a non-zero functional in $L^*$ which
evaluates $f \in L$ at $x$. The Kodaira map sends $x$ to the image of this functional in $\p(L^*)$.
It is a rational map, i.e. defined on a Zariski open subset in $X$. The collection of subspaces in
$\K$ for which the Kodaira map extends to a regular map everywhere on $X$, is a
subsemigroup of $\K$ which we denote by ${\bf K}_{Cart}(X)$.

We show that if $X$ is an irreducible normal, projective variety then the Grothendieck group of
${\bf K}_{Cart}(X)$ is naturally isomorphic to the group of Cartier
divisors on $X$ (Section \ref{sec-Grothendieck-gp-Cartier}). Under this isomorphism the
intersection index of subspaces corresponds to the intersection
index of Cartier divisors. Hence the intersection index of subspaces
introduced in the present paper can be considered as an extension
of the intersection theory of Cartier divisors to
general varieties. For a (not necessarily complete) variety $X$, the
Grothendieck group of the semigroup $\K$, equipped with the
intersection index, can be identified with the direct limit of
groups of Cartier divisors on all the complete birational models of $X$.
(Such limit is studied from an algebraic view point in \cite{French-paper}.)

The semigroup $\K$, equipped with the intersection index, has an
analogue in convex geometry. This analogy is more evident in the
toric case considered by Bernstein and Ku\v{s}nirenko. They
deal with $X= (\c^*)^n$ and the subsemigroup ${\bf K}_{mon} \subset \K$
consisting of the finite dimensional subspaces spanned by
monomials. Each monomial $z^k = z_1^{k_1} \cdots z_n^{k_n}$ on $(\c^*)^n$ corresponds to
a point $k = (k_1, \ldots, k_n)$ in the lattice $\z^n$ (where $z_1, \ldots, z_n$ are
the coordinates on $(\c^*)^n$).
A finite subset $A \subset \z^n$ gives
a subspace $L_A$ in ${\bf K}_{mon}$, namely, $L_A$ is the subspace
spanned by the monomials corresponding to the points in $A$.
It is easy to see that $L_{A+B}
= L_A L_B$ and thus the semigroup ${\bf K}_{mon}$ is isomorphic to
the semigroup of finite subsets of $\z^n$ with respect to the
addition of subsets.

Let $\Delta_1, \ldots, \Delta_n$ be the convex hulls of finite
subsets $A_1, \ldots, A_n \subset \z^n$ respectively.
According to the Bernstein-Ku\v{s}nirenko theorem, the intersection index
$[L_{A_1}, \ldots, L_{A_n}]$ is equal to $n!$ times the mixed
volume of the convex polytopes $\Delta_1, \ldots, \Delta_n$ (when all
the polytopes are the same this was proved by Ku\v{s}nirenko
\cite{Kushnirenko}, and the general case is due to Bernstein
\cite{Bernstein}). Thus the intersection index in the semigroup
${\bf K}_{mon}$ has all the properties of the mixed volume of convex
bodies. In the present paper, we prove that in general for any irreducible
variety $X$, the intersection index in the
semigroup $\K$ also enjoys all the main properties of the mixed
volume. These are listed in Section \ref{sec-mixed-volume-properties}.

The semigroup of finite subsets in $\z^n$ does not have the cancellation
property: the equality $A+C = B+C$ does not imply $A = B$. But the
semigroup of convex bodies in $\r^n$ does have the cancellation
property. In fact, the Grothendieck group of the semigroup of
finite subsets in $\z^n$ is isomorphic to the group of {\it virtual}
integral polytopes in $\r^n$, i.e. the formal differences of
integral convex polytopes (see \cite{Askold-Hilbert-poly}).
The isomorphism is induced by the map which
sends a finite subset $A \subset \z^n$ to its convex hull $\Delta(A)$.
For a given subset $A \subset \z^n$, among the sets $B \subset \z^n$ with $\Delta(A) =
\Delta(B)$, there is a biggest set $\overline{A} = \Delta(A) \cap \z^n$.
The Bernstein-Ku\v{s}nirenko theorem implies
that the space $L_{\overline{A}}$
has the same intersection indices as the subspace $L_A$.

In the semigroup $\K$ of any irreducible variety $X$,
there is a similar operation to taking convex hull. Let us say that
spaces $L_1, L_2 \in \K$ are equivalent if there is $M \in \K$ with
$L_1M = L_2M$. One can show that among all the subspaces equivalent to a
given subspace $L$, there is a biggest subspace $\overline{L} \in \K$.
This subspace $\overline{L}$ is the {\it integral closure} or {\it completion} of $L$: it
consists of all the rational functions which are integral over $L$
(Section \ref{subsec-Grothendieck-group}). The subspace $\overline{L}$ in $\K$ has the same
intersection indices as $L$ (Section \ref{sec-int-index-rational}).

To show the connection with the classical intersection index in
topology and algebraic geometry, in Section \ref{sec-appendix} we give alternative proofs of
the main properties of the intersection index (of subspaces of rational
functions) using the usual topological and algebro-geometric
techniques. In particular, we prove an algebraic analogue of
the Alexandrov-Fenchel inequality from convex geometry using Hodge
theory.

The algebraic analogue of Alexandrov-Fenchel inequality for the
intersection index of Cartier divisors has been proved in
\cite{Burago-Zalgaller} and \cite{Teissier}. Our proof of the
analogous inequality for the intersection index in $\K$ is close to the
proofs in the above papers.

The authors have found surprisingly simple arguments which simultaneously
prove the Alexandrov-Fenchel inequality in convex geometry as well as its
algebraic analogue. These arguments are discussed in detail in a second paper
\cite{Askold-Kiumars-Newton-Okounkov} (preliminary version appeared in \cite{Askold-Kiumars-arXiv}).
The main construction there, corresponds to any subspace $L \in \K$ a convex
body $\Delta(L) \subset \r^n$. The body $\Delta(L)$ is a far-reaching
generalization of the Newton polytope (of a Laurent polynomial). The construction of the body
$\Delta(L)$ depends on a choice of a $\z^n$-valued valuation on the
field of rational functions on the variety $X$.
We show that $n!$ times the volume of $\Delta(L)$ is
equal to the self-intersection index $[L, \ldots, L]$ (see
\cite[Part IV]{Askold-Kiumars-Newton-Okounkov}). The proof is
based on the description of the asymptotic behavior of semigroups
in $\z^{n+1}$. This result is a direct generalization of the
Ku\v{s}nirenko theorem. Unfortunately, in general the identity
$\Delta(L_1) + \Delta(L_2) = \Delta(L_1L_2)$ does not hold, and hence
we do not get a generalization of the Bernstein
theorem. Instead of equality, we have the inclusion $\Delta(L_1) +
\Delta(L_2) \subset \Delta(L_1L_2)$. This inclusion enables us to
give a simple proof of the Alexandrov-Fenchel inequality and its
algebraic analogue.

In a third paper (in preparation) we address the
case when the variety $X$ is equipped with an action of a reductive
algebraic group $G$. We generalize the Bernstein theorem to certain
spherical $G$-varieties. The generalization of the Ku\v{s}nirenko theorem to projective spherical varieties
is due to M. Brion \cite{Brion} (see also \cite{Kazarnovskii} and \cite{Kiritchenko}).
But the Bernstein theorem can be extended only to certain classes of spherical varieties.

After the submission of \cite{Askold-Kiumars-arXiv} to arXiv,
we learned that we were not the only ones to have been working in this direction. Firstly, Okounkov was the first to
define (in passing) a generalization of the notion of Newton polytope
(\cite{Okounkov-Brunn-Minkowski, Okounkov-log-concave}). Although
his case of interest was when $X$ has a reductive group action.
Secondly, Lazarsfeld and Mustata, based on Okounkov's previous
works, and independently of our preprint, came up with closely
related constructions and results \cite{Lazarsfeld-Mustata}. Recently, following
\cite{Lazarsfeld-Mustata}, similar results/constructions have been
obtained for line bundles on arithmetic surfaces \cite{Yuan}.

And a few words about location of the material: the main properties
of mixed volume are recalled in Section \ref{sec-mixed-volume-properties}. In Section
\ref{subsec-B-K} the Bernstein-Ku\v{s}nirenko theorem is discussed and in Section \ref{sec-B-K-rational}
we give a straightforward generalization of this theorem to
rational functions on $(\c^*)^n$. The intersection index for finite
dimensional subspaces of regular functions without common zeros and
on a smooth variety is defined in Section \ref{sec-int-index-regular} and its main
properties are established.
Some of the properties, e.g. multi-linearity
and invariance under the addition of integral elements, are deduced from
the special case when $X$ is a curve. The Alexandrov-Fenchel type
inequality is deduced from the special case when $X$ is a surface.
In Section \ref{sec-int-index-rational} we show that all the results
in Section \ref{sec-int-index-regular} for subspaces of regular functions easily
generalize to the finite dimensional subspaces of rational functions
(possibly with common zeros and poles) on a (possibly non-smooth)
variety $X$. In Section \ref{subsec-regular-Kodaira} we show that for a projective variety
$X$, the Grothendieck group of the subsemigroup ${\bf K}_{Cart}(X)$
is isomorphic to the group of Cartier divisors on $X$ preserving the
intersection index. Finally in Section \ref{sec-appendix} we give proofs of the main
properties of the intersection index using usual techniques from
topology and algebraic geometry.

\section{Motivation and preliminaries} \label{sec-motiv-prelim}
\subsection{Mixed volume of convex bodies and its properties}
\label{sec-mixed-volume-properties} On the space of convex bodies
\footnote{By a convex body in $\r^n$ we mean a compact
convex subset of $\r^n$.} in $\r^n$ one has the following two operations:\\

\noindent 1. (Minkowski sum of convex bodies): For
any two subsets $A, B \subset \r^n$, let $A+B$ denote the set
consisting of all the points $a+b$, $a\in A$, $b \in B$. With this
summation, the collection of subsets of $\r^n$ becomes a commutative
semigroup. In general, this semigroup does not have the cancellation
property: if for subsets $A, B, C$ we have $A+C = B+C$, it does not necessarily
imply that $A=B$. The subset $\{0\}$, is the identity element for
the summation of subsets, i.e. for any $A \subset \r^n$, we have $A+\{0\} =
A$. For two convex bodies $\Delta_1$, $\Delta_2$ in $\r^n$, the set
$\Delta_1 + \Delta_2$ is a convex body, called the {\it Minkowski sum
of} $\Delta_1$ and $\Delta_2$. With this summation, the collection
of convex bodies becomes a commutative semigroup with the cancellation
property. As for any commutative semigroup with the cancellation
property, the collection of formal differences $\Delta_1 - \Delta_2$,
forms an abelian group, where two differences $ \Delta_1
-\Delta_2$ and $\Delta_3 -\Delta_4$ are considered equal if $\Delta_1 +\Delta_4=
\Delta_2 +\Delta_3$. This group is called the group of {\it virtual convex bodies}.\\

\noindent 2. (Multiplication of a convex body by a non-negative
scalar): For a convex body $\Delta$ and $\lambda \geq 0$, the set
$\{ \lambda a \mid a \in \Delta \}$ is a convex body denoted by
$\lambda \Delta$. This scalar multiplication can be extended to
the virtual convex bodies and makes the group of virtual convex bodies
into a real (infinite dimensional) vector space.

\begin{Rem}
It is important in the above definition of scalar multiplication (for
convex bodies) to assume that $\lambda$ is non-negative. If we define
$\lambda \Delta$ for any real $\lambda$ as above, then it will not
agree with the Minkowski sum in a good way. For example, if $\Delta$
contains more than one point and $\lambda > 0$ then $\lambda \Delta
+ (-\lambda)\Delta$ contains more than one point (i.e. is not zero).
\end{Rem}

Convex bodies form a convex cone in the vector space of virtual
convex bodies. \footnote{A cone in a (real) vector space is a subset which is closed under
the addition and multiplication by non-negative scalars.}
On this cone there is a volume function $\Vol$ which
assigns to each convex body $\Delta$, its volume $\Vol(\Delta)$ with
respect to the standard Euclidean measure in $\r^n$ . The function
$\Vol$ is a homogeneous polynomial of degree $n$ on the cone of
convex bodies. That is, its restriction to each finite dimensional
section of the cone is a homogeneous polynomial of degree $n$. More
precisely: let $\r_+^k$ be the positive octant in $\r^k$ consisting
of all ${\bf \lambda} = (\lambda_1,\dots,\lambda_k)$ with
$\lambda_1\geq 0,\dots,\lambda_k\geq 0$. The polynomiality of
$\Vol$ means that for any choice of the convex bodies
$\Delta_1,\dots,\Delta _k$, the function
$P_{\Delta_1,\dots,\Delta_k}$ defined on $\r^k_{+}$ by
$$P_{\Delta_1,\dots,\Delta_k}(\lambda_1,\dots,\lambda_k)=\Vol(\lambda_1
\Delta_1+\dots+
\lambda_k\Delta_k),$$ is a homogeneous polynomial of degree $n$.

By definition the {\it mixed volume of} $V(\Delta_1,\dots,\Delta_n)$
of an $n$-tuple $(\Delta_1,\dots,\Delta_n)$ of convex bodies in $\r^n$ is the
coefficient of the monomial $\lambda_1\dots\lambda_n$ in the
polynomial $P_{\Delta_1,\dots,\Delta_n}$ divided by $n!$.

This definition implies that the mixed volume is the {\it polarization} of
the volume polynomial, that is, it is the unique function on the $n$-tuples
of convex bodies which satisfies the following:
\begin{itemize}
\item[(i)] (Symmetry) $V$ is symmetric with respect to permuting the
bodies $\Delta_1, \ldots, \Delta_n$.
\item[(ii)] (Multi-linearity) It is linear in each argument with
respect to the Minkowski sum. The linearity in the first argument
means that for any choice of convex bodies $\Delta_1'$, $\Delta_1''$,
$\Delta_2,\dots,\Delta_n$ and for any $\lambda_1\geq 0,\lambda_2\geq 0$, we have:
$$ V(\lambda_1\Delta_1'+\lambda_2 \Delta_1'',\dots, \Delta_n)=\lambda_1V(\Delta_1',\dots,
\Delta_n)+\lambda_2V(\Delta_1'',\dots, \Delta_n).$$

\item[(iii)] On the diagonal $V$ coincides with the volume, i.e. for any
convex body $\Delta$ we have $V(\Delta,\dots,
\Delta)=\Vol(\Delta)$.
\end{itemize}

The above three properties characterize the mixed volume: it is the
unique function satisfying (i)-(iii).

By the multi-linearity, the volume and mixed volume functions can be
extended to the vector space of virtual convex bodies.

There are many interesting geometric inequalities known concerning
the mixed volume. The following two inequalities are easy to verify:

\noindent 1) Mixed volume is non-negative, that is, for any $n$-tuple
$\Delta_1, \dots, \Delta_n$ of convex bodies, we have
$$V(\Delta_1,\dots, \Delta_n)\geq 0.$$

\noindent 2) Mixed volume is monotone, that is, for two $n$-tuples
of convex bodies $\Delta'_1\subset \Delta_1,\dots, \Delta'_n\subset
\Delta_n$ we have $$V(\Delta_1,\dots, \Delta_n)\geq
V(\Delta'_1,\dots, \Delta'_n).$$

The next inequality is far more complicated. It is known as the {\it Alexandrov-Fenchel inequality}.

\noindent 3) For any $n$-tuple of convex bodies
$\Delta_1,\dots,\Delta_n \subset \r^n$ one has:
$$V(\Delta_1,\Delta_2, \Delta_3\dots,\Delta_n)^2\geq V(\Delta_1,\Delta_1,
\Delta_3\dots,\Delta_n)V(\Delta_2,\Delta_2,\Delta_3\dots,\Delta_n).$$

Below we mention some formal corollaries of the Alexandrov-Fenchel inequality.
Let us introduce a notation for the repetition
of convex bodies in the mixed volume. Let $2\leq m\leq n$
be an integer and $k_1+\dots+k_r=m$ a partition of $m$ with $k_i \in \n$.
Denote by $V(k_1*\Delta_1,\dots, k_r*\Delta_r,\Delta_{m+1},\dots,\Delta_n)$
the mixed volume of $\Delta_1, \ldots, \Delta_n$, where
$\Delta_1$ is repeated $k_1$ times, $\Delta_2$ is repeated $k_2$ times, etc.
and $\Delta_{m+1},\dots,\Delta_n$ appear once.\\

\noindent 4) With the notation as above, the following inequality holds:
$$V(k_1*\Delta_1,\dots,
k_r*\Delta_r,\Delta_{m+1},\dots,\Delta_n)^m\geq \prod\limits _{1\leq j
\leq r}  V(m*\Delta_j,\Delta_{m+1}\dots,\Delta_n)^{k_j}.$$

\noindent 5) (Generalized Brunn-Minkowski inequality)
Fix convex bodies $\Delta_{m+1},\dots, \Delta_n$. Then
the function $F$ which assigns to a body $\Delta$, the number
$F(\Delta) = V^{1/m}(m*\Delta,\Delta_{m+1},\dots,\Delta_n),$
is concave, i.e. for any two convex bodies $\Delta_1, \Delta_2$ we have
$$F(\Delta_1)+F(\Delta_2)\leq F(\Delta_1+\Delta_2).$$

The $n=m$ case of the above Property (5) was discovered by H. Brunn at the end of the 19th
century. It is called the {\it Brunn-Minkowski inequality}. This inequality was
discovered before the Alexandrov-Fenchel inequality, and its proof is much
simpler than that of the Alexandrov-Fenchel.

\subsection{Mixed volume and the Bernstein-Ku\v{s}nirenko theorem} \label{subsec-B-K}
The beautiful Bernstein-Ku\v{s}nirenko theorem computes the number
of solutions in $(\c^*)^n$ of a sufficiently general system of $n$
equations $P_1=\dots=P_n=0$, where the $P_i$ are Laurent
polynomials, in terms of the Newton polytopes of these polynomials.
In this section we discuss this theorem.

Let us identify the lattice $\z^n$ with {\it Laurent monomials} in
$(\c^*)^n$: to each integral point $k=(k_1, \ldots, k_n) \in \z^n$,
we associate the monomial $z^k=z_1^{k_1}\dots
z_n^{k_n}$. A {\it Laurent polynomial} $P=\sum c_kz^k$ is a finite
linear combination of Laurent monomials with complex coefficients.
The {\it support} $\supp(P)$ of a Laurent polynomial $P$, is the set
of exponents $k$ for which $c_k \neq 0$. We denote the convex hull
of a finite subset $A\subset \z^n$ by $\Delta_A \subset \r^n$.
The {\it Newton polytope} $\Delta (P)$ of a Laurent polynomial $P$
is the convex hull $\Delta_{\supp(P)}$ of its support. To each non-empty
finite subset $A\subset \z^n$ one can associate a finite dimensional vector space
$L_A$ consisting of the Laurent polynomials $P$ with $\supp(P) \subset A$.

\begin{Def}
We say that a property holds for a {\it generic element} of a vector
space $L$, if there is a proper algebraic set $\Sigma$ such that the
property holds for all the elements in $L \setminus \Sigma$.
\end{Def}

\noindent {\bf Problem:} {\it For a given $n$-tuple of non-empty finite subsets
$A_1,\dots,A_n \subset \z^n$, find the number $[L_{A_1}\dots,
L_{A_n}]$ of solutions in $(\c^*)^n$ of a generic system of
equations $P_1=\dots=P_n=0$, where $P_1\in L_1,\dots, P_n\in L_n$
(i.e. find a formula for the number of solutions of a generic
element $(P_1, \ldots, P_n) \in L_{A_1} \times\dots\times
L_{A_n}$)}.

When the convex hulls of the sets $A_i$ are the same and
equal to a polytope $\Delta$, the problem was solved by A. G. Ku\v{s}nirenko (see \cite{Kushnirenko}). He
showed that, in this case, we have
$$[L_{A_1}, \dots,
L_{A_n}]=n! \Vol(\Delta),$$ where $\Vol$ is the standard
$n$-dimensional volume in $\r^n$. In other words, if  $P_1,
\dots,P_n$ are sufficiently general Laurent polynomials with given
Newton polytope $\Delta$, the number of solutions in $(\c^*)^n$
of the system $P_1=\dots=P_n=0$ is equal to $n!\Vol(\Delta)$.

When the convex hulls of the sets $A_i$ are not necessarily the same,
the answer was given by D. Bernstein (see \cite{Bernstein}).
He showed that:
$$[L_{A_1}, \dots,
L_{A_n}]=n! V(\Delta_{A_1}, \ldots, \Delta_{A_n}),$$ where $V$ is the
mixed volume of convex bodies in $\r^n$. In other words, if  $P_1,
\dots,P_n$ are sufficiently general Laurent polynomials with Newton
polytopes $\Delta_1, \ldots, \Delta_n$ respectively, the number of
solutions in $(\Bbb C^*)^n$ of the system $P_1=\dots=P_n=0$ is equal
to $n!V(\Delta_1, \ldots, \Delta_n)$.

Next, let us discuss the Bernstein-Ku\v{s}nirenko theorem from a point of
view which later will allow generalizations.
For two non-zero finite dimensional subspaces $L_1$,$L_2$ of regular
functions in $(\Bbb C^*)^n$, define the product $L_1L_2$ as
the subspace spanned by the products $fg$, where $f\in L_1$, $g\in
L_2$. Clearly the multiplication of monomials corresponds to the addition
of their exponents, i.e. $z^{k}z^{\ell}=z^{k+\ell}.$ This implies
that $L_{A_1}L_{A_2}=L_{A_1+A_2}$.

Among the subspaces of regular functions in $(\c^*)^n$, there is a
natural family consisting of
subspaces stable under the action of the multiplicative
group $(\c^*)^n$. It is easy to see that any such non-zero subspace is of the form $L_A$ for some non-empty
finite subset $A\subset \z^n$ of monomials.
As mentioned above, $L_AL_B=L_{A+B}$ which implies that the
collection of stable subspaces is closed under the product of subspaces.

For a finite set $A \subset \z^n$, let $\overline{A} = \Delta_A \cap \z^n$.
To a subspace $L_A$ we can associate the bigger subspace $L_{\overline{A}}$.
From the Bernstein-Ku\v{s}nirenko theorem it follows that
for any $(n-1)$-tuple of finite subsets $A_2, \ldots A_n \in \z^n$ we have
$$[L_A, L_{A_2}, \ldots, L_{A_n}] = [L_{\overline{A}}, L_{A_2}, \ldots, L_{A_n}].$$
That is, (surprisingly!) enlarging $L_A \mapsto L_{\overline{A}}$
does not change any of the intersection indices. In other words,
in counting the number of solutions of a system, instead of
the support of a polynomial, its convex hull plays the main role. We
denote the subspace $L_{\overline{A}}$ by $\overline{L_A}$, and call
it the {\it completion of $L_A$}.

Recall that the semigroup of convex bodies with Minkowski sum has the cancellation
property. It implies that the following cancellation property holds
for the finite subsets of $\z^n$: if for finite subsets $A, B, C \in
\z^n$ we have $\overline{\overline{A} + \overline{C}} =
\overline{\overline{B} + \overline{C}}$ then $\overline{A} =
\overline{B}$. And the same cancellation property holds for the
corresponding semigroup of subspaces $L_A$. That is, if
$\overline{\overline{L_A} ~\overline{L_C}} =
\overline{\overline{L_B}~\overline{L_C}}$ then
$\overline{L_A} = \overline{L_B}$.\\

The Bernstein-Ku\v{s}nirenko theorem defines an intersection index
on $n$-tuples of subspaces of type $L_A$ (for non-empty finite subsets $A
\subset \z^n$), and computes it in terms of the mixed volume of polytopes.
It follows that this intersection index should enjoy the same properties as the mixed volume,
namely: 1) positivity, 2)
monotonicity, 3) multi-linearity, 4) the Alexandrov-Fenchel inequality
and its corollaries. Moreover, we have 5) the subspaces $L_A$ and
$L_{\overline{A}}$ have the same intersection indices.

\section{Bernstein-Ku\v{s}nirenko theorem for rational functions and virtual polytopes}
\label{sec-B-K-rational}
Let $R = P/Q$ be a rational function on $(\c^*)^n$.

\begin{Def} We define the {\it virtual Newton polytope $\Delta_R$ of $R$} to be
$\Delta_P - \Delta_Q$, the formal difference of the Newton polytopes of
the polynomials $P$ and $Q$.
\end{Def}

The above definition is well-defined, that is if $P_1/Q_1=P_2/Q_2$ then
$\Delta_{P_1} + \Delta_{Q_2} = \Delta_{P_2} + \Delta_{Q_1}$ and hence
$\Delta_{P_1} - \Delta_{Q_1} = \Delta_{P_2} -\Delta_{Q_2}$.

Unlike the case of Laurent polynomials and Newton polytopes,
the family of rational functions with a given virtual polytope $\Delta
= \Delta' - \Delta''$ usually can not be described by finitely many
parameters.

\begin{Ex}
Let $n=1$ and take $\Delta = \{0\}$. For any $k>0$ take two
polynomials (in one variable) $P$ and $Q$ such that
$\deg(P)=\deg(Q)=k$ and $P(0) \neq 0$ and $Q(0) \neq 0$. Then the
Newton segments of $P$ and $Q$ are both equal to the segment $[0,k]$
and hence the rational function $R = P/Q$ has $\{0\}$ as its virtual
segment. For each $k$, the rational function $R=P/Q$ depends on
$2k+1$ parameters and $k$ can be chosen to be arbitrarily large.
\end{Ex}

\begin{Def}
For two non-empty finite subsets $A, B \subset \z^n$, let the formal quotient
$L_A/L_B$ denote the collection of all the rational functions $R=P/Q$
where $P, Q$ are Laurent polynomials with $P \in L_A$ and $Q \in
L_B$. We call $L_A/L_B$ a {\it virtual subspace}.
\end{Def}

Take finite subsets $A_i, B_i \subset \z^n$, $i=1, \ldots, n$.
We want to associate a number
$$[L_{A_1}/L_{B_1}, \ldots, L_{A_n}/L_{B_n}]$$ to the virtual subspaces
$L_{A_i}/L_{B_i}$ which computes the intersection number of the {\it principal divisors}
of generic functions $R_i \in L_{A_i}/L_{B_i}$:
take a partition of $\{1, \ldots, n\}$ into two subsets $I=\{i_1, \ldots,i_r\}$ and
$J=\{j_1,\ldots,j_s\}$.
To this partition associate the number
$$N(I) = [L_{A_{i_1}}, \ldots, L_{A_{i_r}}, L_{B_{j_1}}, \ldots, L_{B_{j_s}}].$$
That is, $N(I)$ is the number of solutions $x \in (\c^*)^n$ of a generic system
$P_{i_1} = \ldots = P_{i_r} = Q_{j_1} = \ldots = Q_{j_s} = 0$, with
$P_{i_k} \in L_{A_{i_k}}$ and $Q_{j_k} \in L_{B_{j_k}}$.

\begin{Def}
For virtual subspaces $L_{A_1}/L_{B_1}, \ldots, L_{A_n}/L_{B_n}$, define the
number $[L_{A_1}/L_{B_1}, \ldots, L_{A_n}/L_{B_n}]$ by
$$[L_{A_1}/L_{B_1}, \ldots, L_{A_n}/L_{B_n}] = \sum_{I \subset
\{1,\ldots,n\}} (-1)^{n-|I|}N(I).$$
\end{Def}

If $B_1 = \ldots = B_n = \{0\}$, then each virtual subspace
$L_{A_i}/L_{B_i}$ is equal to $L_{A_i}$ and the number
$[L_{A_1}/L_{B_1}, \ldots, L_{A_n}/L_{B_n}]$ is just the number of
solutions in $(\c^*)^n$ of a generic system $P_1 = \ldots = P_n = 0$
of Laurent polynomials $P_i \in L_{A_i}$. Unlike $[L_{A_1}, \ldots,
L_{A_n}]$, the number $[L_{A_1}/L_{B_1}, \ldots, L_{A_n}/L_{B_n}]$
can be negative. But it is still multi-linear, and symmetric i.e.
invariant under the permutation of its arguments.

The following theorem is the extension of Bernstein-Ku\v{s}nirenko theorem to
rational functions and virtual polytopes.
For each $i$, let
$\Delta_i$ denote the virtual polytope $\Delta_{A_i} - \Delta_{B_i}$.

\begin{Th}[Bernstein-K\v{u}snirenko for rational functions]
$$[L_{A_1}/L_{B_1}, \ldots, L_{A_n}/L_{B_n}] = n!V(\Delta_1, \ldots, \Delta_n).$$
\end{Th}
\begin{proof}
By multi-linearity of mixed volume for virtual convex bodies, we have:

\begin{eqnarray} \label{eqn-rational-Bernstein-Kushnirenko}
n!V(\Delta_1, \ldots, \Delta_n) &=& n!V(\Delta_{A_1} - \Delta_{B_1},
\ldots, \Delta_{A_n} - \Delta_{B_n}), \cr &=& n! \sum_{I \subset
\{1, \ldots, n\}} (-1)^{n-|I|} V(\Delta_{A_{i_1}}, \ldots,
\Delta_{A_{i_r}}, \Delta_{B_{j_1}}, \ldots, \Delta_{B_{j_s}}), \cr
\end{eqnarray}
where as above $I = \{i_1, \ldots, i_r\}$ and $J = \{j_1, \ldots,
j_s\} = \{1, \ldots, n\} \setminus I$. Now by the
Bernstein-Ku\v{s}nirenko theorem
$$n!V(\Delta_{A_{i_1}}, \ldots, \Delta_{A_{i_r}},
\Delta_{B_{j_1}}, \ldots, \Delta_{B_{j_s}}) =
[L_{A_{i_1}}, \ldots, L_{A_{i_r}}, L_{B_{j_1}}, \ldots, L_{B_{j_s}}],$$ and from
(\ref{eqn-rational-Bernstein-Kushnirenko}) we have
$$n!V(\Delta_1, \ldots, \Delta_n) = \sum_{I \subset \{1, \ldots, n\}} (-1)^{n-|I|}
[L_{A_{i_1}}, \ldots, L_{A_{i_r}}, L_{B_{j_1}}, \ldots,
L_{B_{j_s}}],$$ which is just the definition of $[L_{A_1}/L_{B_1},
\ldots, L_{A_n}/L_{B_n}]$.
\end{proof}

\section{An intersection index for subspaces of regular functions} \label{sec-int-index-regular}
\subsection{Semigroup of subspaces of an algebra of functions}  \label{subsec-1}
In this section we define the intersection index of a collection of
finite dimensional subspaces $L_1, \ldots, L_n$ of rational functions on an
$n$-dimensional variety $X$, and prove its main properties. Roughly
speaking, this intersection index is the number of solutions of a
generic system $f_1=\ldots=f_n$ where $f_i \in L_i$.

For convenience, we restrict ourselves to regular functions and
smooth varieties. Later, we will generalize the situation to rational
functions on arbitrary varieties.

We start with some general definitions.  {\it A set equipped with an algebra
of functions} is a set $X$, with an algebra $R(X)$ consisting of complex
valued functions on $X$, and containing all the constants. To a pair
$(X, R(X))$ one can associate the set  $VR(X)$ whose elements are
vector subspaces of $R(X)$.

Any subspace $L \in VR(X)$ gives rise to a natural map
$\tilde{\Phi}_L: X \to L^*$, where $L^*$ denotes the vector space
dual of $L$, as follows: for $x \in X$ let $\tilde{\Phi}_L(x) \in
L^*$ be the linear function defined by
$$\tilde{\Phi}_L(x)(f) = f(x),$$ for all $f \in L$.

There is a natural multiplication in $VR(X)$:

\begin{Def}
For any two subspaces $L_1, L_2\subset R(X)$ define the {\it product
$L_1L_2$} to be the linear span of the functions $fg$, where $f\in L_1$
and $g\in L_2$. With this product the set $VR(X)$ becomes a
commutative semigroup. The space $L_1L_2$ can be considered as a
factor of the tensor product $L_1 \otimes L_2$ and there is a
natural projection $\pi: L_1 \otimes L_2 \to L_1L_2$: for $v \in
L_1 \otimes L_2$ where $v = \sum_i f_i \otimes g_i$ with $f_i \in
L_1$, $g_i \in L_2$, define $\pi(v) = \sum_i f_ig_i$. The projection
$\pi$ is onto but can have non-zero kernel.
\end{Def}

Let us say that a subspace $L$ {\it has no base locus on $X$}, if
for each $x\in X$ there is a function $f\in L$ with $f(x)\neq
0$. The following is easy to verify:

\begin{Prop} \label{1.1}
Let $L_1,L_2 $ be vector subspaces of $R(X)$. If  $L_1, L_2 $ are
finite dimensional (respectively, if $L_1, L_2 $ have no
base locus on $X$), then the space $L_1L_2$ is finite dimensional
(respectively, has no base locus on $X$).
\end{Prop}


According to Proposition \ref{1.1}, the finite dimensional
subspaces of $R(X)$ with no base locus, form a
semigroup in $VR(X)$ which we will denote by $KR(X)$.

Assume that $Y \subset X$ and that the restriction of every function
$f\in R(X)$ to $Y$ belongs to an algebra $R(Y)$.
We will denote the restriction of a subspace $L \subset R(X)$ to $Y$
again by $L$. Clearly, if  $L\in KR(X)$ then $L\in KR(Y)$.

In this paper we will not use general sets equipped with algebras of
functions. Instead, the following case plays the main role.
\begin{Def} Let $X$ be a variety
and let $R(X)=\mathcal{O}(X)$ be the algebra of regular functions on
$X$. In this case to simplify the notation we will not mention the
algebra $R(X)$ explicitly and the semigroup $KR(X)$ will be denoted by
${\bf K}_{reg}(X)$.
\end{Def}

As above any subspace $L \in {\bf K}_{reg}(X)$ gives a natural map
$\tilde{\Phi}_L: X \to L^*$. Since by assumption $L$ has no base
locus, for all $x \in X$, we have $\tilde{\Phi}_L(x) \neq 0$.
\begin{Def}[Kodaira map]
For any $x \in X$, let $\Phi_L(x)$ be the point in $\p(L^*)$ represented by
$\tilde{\Phi}_L(x) \in L^*$. We call $\Phi_L: X \to \p(L^*)$, the {\it Kodaira map of the subspace $L$}.
It is a morphism from $X$ to $\p(L^*)$.
\end{Def}

Fix a basis $\{f_1, \ldots f_d\}$ for $L$. One verifies that,
in the homogeneous coordinates in $\p(L^*)$ corresponding to the dual basis to the $f_i$,
the map $\Phi_L$ is given by $$\Phi_L(x) = (f_1(x) : \cdots: f_d(x)).$$

Finally let us define the notion of an integral element.
\begin{Def} \label{def-integral-L}
Let us say that a regular function  $f \in \mathcal{O}(X)$ is {\it
integral over a subspace $ L\in {\bf K}_{reg}(X)$}, if $f$ satisfies
an equation $$f^m+a_1f^{m-1}+\dots +a_m=0,$$ where  $m>0$ and $a_i
\in L^i$, for each $i=1,\ldots, m$.
\end{Def}

\subsection{Preliminaries on algebraic varieties} \label{subsec-2}
In this section, we discuss some well-known results which we need to define an intersection
index in the semigroup ${\bf K}_{reg}(X)$. We will need the particular cases of the
following results: 1) An algebraic variety has a {\it finite}
topology. 2) There are finitely many topologically different
varieties in an algebraic family of algebraic varieties. 3) In such a family
the set of parameters, for which the corresponding members have the same topology,
is a complex semi-algebraic subset in the  space of parameters.
4) A complex semi-algebraic subset in a vector space either covers almost
all of the space, or covers only a very small part of it.

Let us give precise statements of these results and their
particular cases we will use. Let $X$, $Y$ be algebraic varieties and let
$\pi: X \to Y$ be a morphism. Consider the family of
algebraic varieties  $X_y=\pi^{-1}(y)$ parameterized by the points $y\in
Y$. The following is well-known.

\begin{Th} \label{th-strat}
Each variety $X_{y}$ has the homotopy type of a
finite $CW$-complex. There is a finite  stratification of the
variety $Y$ into complex semi-algebraic strata $Y_{\alpha}$, such
that for any two points $y_1$, $y_2$ in the same stratum $Y_{\alpha}$,
the varieties $X_{y_1}$, $X_{y_2},$ are homeomorphic. (In particular, in
the family $X_y$ there are only finitely many topologically
different varieties.)
\end{Th}

When $X$, $Y$ are real algebraic varieties and $\pi:X\to Y$ is a
(real) morphism, a similar statement holds. There are also extensions of
the above to some other cases (see \cite{Dries}).
We will need only the following simple corollary of this theorem.

Let $L_1,\dots, L_n $ be finite dimensional subspaces
of  regular functions on an $n$-dimensional
algebraic variety $X$. For ${\bf f}$ $=(f_1,\dots,f_n) \in \bf L$ $= L_1\times\dots\times
L_n$, let $X_{\bf f}$ denote the subvariety of $X$, defined by the
system of equations $f_1=\dots=f_n=0$. In the space $\bf L$
of parameters consider the subset $\bf F$ consisting of all
the parameters $\bf f$ for which the set $X_{\bf f}$ consists of isolated points only.
\begin{Cor} \label{2.1}
1) For every ${\bf f} \in {\bf F}$ the set $X_{\bf f}$ is finite.
2) Denote the number of points in $X_{\bf f}$ by $k({\bf f})$, then there is a
finite stratification of the set ${\bf F}$ with complex
semi-algebraic strata $Y_\alpha$ such that the
function $k({\bf f})$ is constant on each stratum. In particular, the subset ${\bf
F}_{\max} \subset {\bf F}$ where $k({\bf f})$ attains its maximum is
a complex semi-algebraic set.
\end{Cor}

The above corollary can be proved without using Theorem
\ref{th-strat}. The semi-algebraicity of the set ${\bf F}$ and its
subsets ${\bf F}_m = \{{\bf f} \in {\bf F} \mid k({\bf f}) = m \}$
follow from the complex version of the Tarski theorem. An analogous fact in
the real case can also be proved using the (real) Tarski theorem. An
elementary proof of the Tarski theorem and its complex version can be
found in \cite{Burda-Khovanskii}.

We will need the following simple property of complex semi-algebraic
sets.
\begin{Prop} Let $F\subset L$ be a complex
semi-algebraic subset in a (complex) vector space $L$. Then either there is an
algebraic hypersurface  $\Sigma\subset L$ which contains $F$,
or $F$ contains a Zariski open set  $U\subset L$.
\end{Prop}

We will use this proposition in the following form.
\begin{Cor} \label{2.2}
Let $F\subset L$ be a complex semi-algebraic subset in a (complex) vector
space  $L$. In either of the following cases, $F$ contains a Zariski
open subset  $U\subset L$: 1) $F$ is an everywhere dense  subset of
$L$, or 2) $F$ does not have zero measure.
\end{Cor}

\subsection{An intersection index in semigroup ${\bf K}_{reg}(X)$} \label{subsec-3}
\begin{Def} Let $X$ be a smooth (not necessarily connected)
complex $n$-dimensional variety and let $L_1,\dots, L_n \in {\bf K}_{reg}(X)$.
The {\it intersection index} $[L_1,\dots, L_n]$ is the maximum number of roots of a
system $f_1=\dots=f_n=0$ over all the points $\bf f$ $=(f_1,\dots,f_n)\in
L_1\times\dots\times L_n=\bf L$, for which the corresponding system
has finitely many solutions.
\end{Def}

By Corollary \ref{2.1} the maximum is attained and the previous definition
is well-defined.
The following is a straightforward corollary of the definition.
\begin{Th}[Obvious properties of the intersection
index] \label{3.1} Let $L_1, \ldots, L_n \in {\bf K}_{reg}(X)$, then:
1) $[L_1,\dots,L_n]$ is a symmetric function of
the $n$-tuple $L_1,\dots, L_n$ (i.e. takes the same value under a
permutation of $L_1,\dots,L_n$). 2) The intersection
index is monotone (i.e. if $L'_1\subseteq L_1,\dots, L'_n\subseteq
L_n$, then $[L_1,\dots,L_n]\geq [L'_1,\dots,L'_n])$. 3) The
intersection index is non-negative (i.e. $[L_1,\dots,L_n]\geq 0$).
\end{Th}

Let $X$ be a smooth (not necessary connected) complex $n$-dimensional variety, and
let $L_1,\dots, L_k \in {\bf K}_{reg}(X)$. Put ${\bf L}(k) =L_1\times \dots \times L_k$.

\begin{Prop} \label{3.2}
There is a non-empty Zariski open set $\bf U$
in ${\bf L}(k)$ such that either for any point $\bf f$ $=(f_1,\dots,f_k)$ in
$\bf U$ the system of equations  $f_1=\dots =f_k=0$ has no solution in $X$, or the system is consistent
\footnote{A system is consistent if its set of solutions is non-empty.} and
non-degenerate (i.e. at every root of the system the differentials
$df_1,\dots,df_k$ are linearly independent).
\end{Prop}

\begin{proof} Fix a basis $\{g_{ij}\}$ for each of the spaces $L_i$.
Consider all the $k$-tuples
${\bf g}_{\bf j} =( g_{1j_1},\dots, g_{kj_k})$, where ${\bf j}=(j_1,\dots,j_k)$,
containing exactly one vector from each of the bases $\{g_{ij}\}$ for the $L_i$.
Denote by $V_{\bf j}$ the Zariski open subset in $X$ defined by the system of
inequalities $ g_{1,j_1}\neq 0 ,\dots, g_{k,j_k}\neq 0$. (Note that since $X$ can be reducible,
the set $V_{\bf j}$ might be empty.) Since $L_1,
\ldots, L_k$ are in ${\bf K}_{reg}(X)$, the union of the sets
$V_{\bf j}$ equals $X$. When the set $V_{\bf j}$ is non-empty, rewrite the system $f_1=\dots
=f_k=0$ as follows: represent each function $f_i$ in the form
$f_i=\overline{f_i} +c_ig_{ij_i}$, where $\overline{f_i}$ belongs
to the linear span of the $g_{ij}$ excluding $g_{ij_i}$. Now, in
$V_{\bf j}$, the system can be rewritten as  $\frac{\overline{f_1}}{g_{1,j_1}}=-c_1, \dots, \frac
{\overline{f_k}}{g_{k,j_k}}=-c_k$. According to Sard's theorem, for
almost all the $\bf c =(c_1,\dots, c_k)$ the system is
non-degenerate. Denote by $W_{\bf j}$ the subset in ${\bf L}(k)$,
consisting of all the $\bf f$ for which the system $f_1=\dots =f_k=0$ is
non-degenerate in $V_{\bf j}$. The set $W_{\bf
j}$ splits into the union of the complex semi-algebraic subsets $W_{\bf
j}^1$ and $W_{\bf j}^2$, containing the consistent or
non-consistent systems respectively.
Thus by Corollary \ref{2.2}, exactly one of the sets $W_{\bf j}^1$ or $W_{\bf j}^2$ contains a non-empty Zariski open
subset $\bf U_{\bf j}$. The intersection $\bf U$ of the sets  $\bf
U_{\bf j}$ is a Zariski open subset in ${\bf L}(k)$ which satisfies
all the requirements of the proposition.
\end{proof}

\begin{Prop} \label{3.3}
The number of isolated roots, counted with multiplicity, of a
system $f_1=\dots=f_n=0$, where $f_i\in L_i$, is less than or equal to
$[L_1,\dots,L_n]$.
\end{Prop}
\begin{proof}
Suppose $A$ is a subset of the isolated roots of the system such that
$k(A)$, the sum of multiplicities of the roots in $A$, is bigger than
$[L_1, \ldots, L_n]$. According to Proposition \ref{3.2} one can
slightly perturb the system to make it non-degenerate. Under such a
perturbation the roots belonging to the set $A$  will split into
$k(A)>[L_1,\dots,L_n]$ simple roots while all other roots are also
simple. But by Corollary \ref{2.1} the number of these roots can not
be bigger than $[L_1,\dots,L_n]$. The contradiction proves the
proposition.
\end{proof}

Now we prove that if a system of equations is in general position
then in Proposition \ref{3.3}, instead of inequality we have an
equality. As before let ${\bf L} = L_1\times \dots \times L_n$.

\begin{Prop} \label{3.4}
There is a non-empty Zariski open set $\bf U$
in $\bf L$ such that for any point ${\bf f} =(f_1,\dots,f_n)$ in
$\bf U$ the system of equations  $f_1=\dots =f_n=0$ on $X$ is
non-degenerate and has exactly $[L_1,\dots,L_n]$ solutions.
\end{Prop}
\begin{proof} Firstly, if the number of isolates roots of a system is equal to
$[L_1,\dots,L_n]$ then it is non-degenerate (otherwise its
number of roots, counting with multiplicity, would be bigger than
$[L_1,\dots, L_n]$ which is impossible by Proposition \ref{3.3}). So
there must be a non-degenerate system which has $[L_1,\dots,L_n]$ roots.
Secondly, any system which is close enough to this system has
exactly the same number of isolated roots and almost all such
systems are non-degenerate. So the set of non-degenerate systems
which have exactly $[L_1,\dots,L_n]$ roots is not a set of
measure zero. But since this set is complex semi-algebraic, by
Corollary \ref{2.2}, it should contain a (non-empty) Zariski open subset.
\end{proof}

Let $Y$ be an $m$-dimensional \qp smooth subvariety of $X$. For any
$m$-tuple of subspaces $L_1,\dots,L_m \in {\bf K}_{reg}(X)$ let
$[L_1,\dots,L_m]_Y$  denote the intersection index of the restrictions of
these subspaces to $Y$.

Consider an $n$-tuple  $L_1,\dots,L_n\in {\bf K}_{reg}(X)$. As before, for $k \leq n$
put ${\bf L}(k) = L_1\times\dots\times L_k$.
According to Proposition \ref{3.2}, there is a non-empty Zariski open subset
${\bf U}(k)$ in ${\bf L}(k)$ such that for ${\bf f}(k)
=(f_1,\dots,f_k)\in {\bf U}(k)$, either the system $f_1=\dots=f_k=0$ has no
solutions or it defines a smooth subvariety $X_{{\bf f}(k)}$
in $X$.

\begin{Th} \label{3.5}
If the system $f_1=\dots=f_k=0$ has no solution then $[L_1,\dots,L_n]=0$.
If the system is consistent, then the following holds:
1) For any point ${\bf f}(k)\in
{\bf U}(k)$ we have:
\begin{equation} \label{eqn-3.5}
[L_1,\dots L_n]_X\geq [L_{k+1},\dots L_n]_{X_{{\bf f}(k)}}.
\end{equation}
2) Moreover, there is a Zariski open subset ${\bf V}(k) \subset {\bf U}(k)$, such
that for any ${\bf f}(k)\in {\bf V}(k)$, the inequality (\ref{eqn-3.5})
is in fact an equality.
\end{Th}
\begin{proof} The statement about non-consistent systems is obvious.
Let us prove the other statements. 1) If (\ref{eqn-3.5})
does not hold for a point ${\bf f}(k)$, then there are  $f_{k+1}\in
L_{k+1}, \dots, f_{n}\in L_{n}$  such that the system
$f_1=\ldots=f_n=0$ has more isolated solution on $X$ than the
intersection index $[L_1, \dots,L_n ]$, which is impossible. 2)
According to Proposition \ref{3.2} the collection of systems ${\bf f}
=(f_1,\dots,f_ n)\in {\bf L}$ which belong to the Zariski open set
$\bf U$ in Proposition \ref{3.4} and for which the subsystem
$f_1=\dots=f_k=0$ is non-degenerate, contains a Zariski open set
$\bf V \subset \bf U$. Let $\pi:{\bf L} \to {\bf L}(k)$ be
the projection $(f_1,\dots,f_n) \mapsto (f_1,\dots,f_k)$. Now we can
take ${\bf V}(k)$ to be any (non-empty) Zariski open subset in ${\bf L}(k)$
contained in $\pi(\bf V)$.
\end{proof}

Theorem \ref{3.5} allows us to reduce the computation of the
intersection index on a higher dimensional smooth \qp variety to the
computation of the intersection index on a lower dimensional smooth
subvariety. As we will see, it is not hard to establish the main properties of the
intersection index for affine curves. Using Theorem \ref{3.5} we
then easily obtain the corresponding properties for the intersection
index on smooth varieties of arbitrary dimension.

\subsection{Preliminaries on algebraic curves} \label{subsec-4}
Here we recall some basic facts about algebraic curves which we
will use later. Let $X$ be a smooth \qp curve (not necessarily
irreducible or complete).

\begin{Th}[Normalization of algebraic curves]
There is a unique (up to isomorphism) smooth compactification
$\overline{X}$ of $X$. The complement  $A = \overline{X} \setminus
X$, is a finite set, and any regular function on $X$ has a
meromorphic extension to $\overline{X}$.
\end{Th}

One can find a proof of this classical result in most of the text
books in algebraic geometry (e.g. \cite[Chapter 1]{Hartshorne}).
This theorem allows us to find the number of zeros of a regular
function $g$ on $X$ with a prescribed behavior at infinity,
i.e. $\overline X \setminus X$. Indeed if  $g$ is not identically
zero on some irreducible component of the curve $X$, then the order
$\ord_a(g)$ of its meromorphic extension at a point $a\in \overline
X$ is well-defined. The function $g$ on the projective curve
$\overline X$ has the same number of roots
as the number of poles (counting with multiplicity).
Thus we have the following:
\begin{Prop} \label{4.1}
For every  regular  function $g$ on
a smooth \qp curve $X$ (which is not identically
zero on any irreducible component of $X$) the number of roots,
counting with multiplicity, is equal to  $- \sum_{a\in A}
\ord_a(g)$, where $\ord_a(g)$ is the order of the
meromorphic extension of the function $g$ to $\overline X$ at the point $a$.
\end{Prop}

\subsection{Intersection index in semigroup ${\bf K}_{reg}(X)$ of an
algebraic curve $X$} \label{subsec-5}
Let $L\in {\bf K}_{reg}(X)$ and let $B=\{ f_i\}$ be a basis for $L$ such that none of the $f_i$
are identically zero on any irreducible component of the curve $X$. For
each point $a\in A=\overline X\setminus X$ denote by $\ord_a(L)$ the
minimum, over all functions in $B$, of the numbers $\ord_a(f_i)$.
Clearly, for every $g \in L$ we have $\ord_a(g) \geq \ord_a(L)$. The
collection of functions $g\in L$ whose order at the point $a$ is
strictly bigger than $\ord_a(L)$ form a proper subspace $L_a$ of
$L$.

\begin{Def} Define the {\it degree} of a subspace $L\in {\bf K}_{reg}(X)$
to be $$\sum\limits _{a\in A} -\ord_a(L),$$ and denoted it by $\deg(L)$.
\end{Def}

For each irreducible component $X_i$ of the curve $X$, denote the subspace of
$L$, consisting of all the functions identically equal to zero on
$X_i$, by $L_{X_i}$.  The subspace $L_{X_i}$ is a proper subspace of
$L$ because $L$ has no base locus.

The following is a corollary of Proposition \ref{4.1}.
\begin {Prop} \label{5.1}
If a function  $f\in L$ does not
belong to the union of the subspaces $L_{X_i}$, then $f$ has
finitely many roots on $X$. The number of roots of the function
$f$ (counting with multiplicity) is less than or equal to $\deg(L)$.
If $f$ is not in the union of
the subspaces $L_a$, $a\in A$, then the equality holds.
\end{Prop}

\begin{Prop} \label{5.2}
For $L, M \in {\bf K}_{reg}(X)$ we have
$[L]+[M]=[LM].$
\end{Prop}
\begin{proof} For any point $a\in A$ and any two functions $f \in L$,
$g \in M$ the identity $\ord_a(f) + \ord_a(g) = \ord_a (fg)$ holds. Thus
we have $\ord_a(L) + \ord_a(M)= \ord_a (LM)$ which implies
$\deg(L) +\deg(M) = \deg(LM)$, and hence $[L]+[M]=[LM]$.
\end{proof}

Consider the map $-\Ord$ which associates to a subspace $L\in {\bf K}_{reg}(X)$
an integer valued function on the set $A$: the value of $-\Ord(L)$
at $a\in A$ equals $-\ord_a(L)$. The map $-\Ord$ is a homomorphism from
the multiplicative semigroup ${\bf K}_{reg}(X)$ to the additive group of
integer valued functions on the set $A$. Clearly the number $[L]$ can be
computed in terms of the homomorphism $-\Ord$ because $[L]= \deg(L)=
\sum_{a\in A} -\ord _aL $.

\begin{Prop} \label{5.3}
Assume that a regular function $g$
on the curve $X$ is integral over a
subspace $L\in {\bf K}_{reg}(X)$. Then at each point $a\in A$ we have
$$\ord_a g \geq \ord_a L.$$
\end{Prop}
\begin{proof} Let $g^n +f_1g^{n-1}+\dots +f_n=0$ where
$f_i \in L^i$. Suppose $\ord_a g =k < \ord_a L$.
Since $g^n = -f_1g^{n-1} - \cdots - f_n$ we have
$nk = \ord_a(g^n) \geq \min\{\ord_a (f_1g^{n-1}), \ldots, \ord_a(f_n)\}$.
That is, for some $i$, $nk \geq \ord_a(f_i) + k(n-i)$.
But for every $i$, $\ord_a(f_i g^{n-i}) = \ord_a(f_i) + \ord_a(g^{n-i})
> i \cdot \ord_a(L) + k(n-i) > nk$. The contradiction proves the claim.
\end{proof}

\begin{Cor} \label{5.4}
Assume that a regular function $g$ on the curve $X$ is integral over
a subspace  $L\in {\bf K}_{reg}(X)$. Consider the subspace $M \in
{\bf K}_{reg}(X)$ spanned by $g$ and $L$. Then: 1) At each point
$a\in A$ the equality $\ord_a(L)=\ord_a(M)$ holds. 2) $[L]=[M]$. 3)
For each subspace $N \in {\bf K}_{reg}(X)$ we have $[LN]=[MN]$.
\end{Cor}

\subsection{Properties of the intersection index which can be deduced from the curve case}
\label{subsec-6}
\begin{Th}[Multi-linearity] \label{6.1}
Let $L_1', L_1'', L_2, \ldots, L_n \in {\bf K}_{reg}(X)$
and put $L_1= L_1'L_1''$. Then
$$[L_1,\dots,L_n]=[L'_1,\dots,L_n]+[L''_1,\dots,L_n].$$
\end{Th}
\begin{proof} Consider three $n$-tuples $(L'_1,\dots,L_n)$,
$(L''_1,\dots,L_n)$ and $(L'_1L''_1,\dots,L_n)$ of elements of the
semigroup ${\bf K}_{reg}(X)$. If a generic system $f_2=\cdots=f_n=0$
where $f_2\in L_2,\dots, f_n\in L_n$ is non-consistent then the three indices
appeared in the theorem are equal to zero. Thus the theorem holds in this case.
Otherwise, according to Theorem \ref{3.5} there is an $(n-1)$-tuple $f_2\in L_2,\dots, f_n\in L_n$, such that
the system $f_2=\dots = f_n=0$ is non-degenerate and defines a curve
$Y\subset X$ such that $[L'_1,\dots,L_n]=[L_1']_Y$,
$[L''_1,\dots,L_n]=[L_1'']_Y$ and
$[L'_1L''_1,\dots,L_n]=[L_1'L''_1]_Y$. Using Proposition \ref{5.2}
we obtain $[L_1'L''_1]_Y= [L_1']_Y+[L''_1]_Y$ and theorem is proved.
\end{proof}

\begin{Cor} \label{th-index-kL}
Let $L_1, L_2, \ldots, L_n \in {\bf K}_{reg}(X)$, then for any $k > 0$ we have
$$[L_1^k, L_2, \ldots, L_n] = k[L_1, L_2, \ldots, L_n].$$
\end{Cor}

By the above corollary, given some $k>0$, if we know the
intersection index of the space $L_1^k$ and any $(n-1)$-tuple of
subspaces $L_2, \ldots, L_n$, we can recover the intersection index of
$L_1$ and $L_2, \ldots L_n$. The computation of the intersection
index of $L_1^k$ might be easier since it contains more functions.

\begin{Th}[Addition of integral elements]
\label{6.2}
Let $L_1 \in {\bf K}_{reg}(X)$ and let
$M_1\in {\bf K}_{reg}(X)$ be a subspace spanned by $L_1 \in {\bf K}_{reg}(X)$ and
a regular function $g$ which is integral over $L_1$. Then for any
$(n-1)$-tuple $L_2,\dots,L_n\in {\bf K}_{reg}(X)$ we have
$$[M_1,L_2,\ldots,L_n]=[L_1,L_2,\ldots,L_n].$$
\end{Th}
\begin{proof} Consider two $n$-tuples $(L_1,L_2,\dots,L_n)$,
$(M_1,L_2,\dots,L_n)$ of ${\bf K}_{reg}(X)$. If a generic system $f_2= \cdots =f_n=0$,
where $f_2 \in L_2, \dots, f_n \in L_n$ is non-consistent, then the two indices
appearing in the theorem are equal to zero. Thus the theorem holds this case.
Otherwise, according to Theorem \ref{3.5} there is an $(n-1)$-tuple $(f_2, \ldots, f_n)$, $f_i \in L_i$,
such that the system $f_2=\dots = f_n=0$ is non-degenerate and
defines a curve $Y\subset X$ with
$[L_1,L_2,\dots,L_n]=[L_1]_Y$ and $[G_1,L_2 \dots,L_n]=[G_1]_Y$. Using
Corollary \ref{5.4} we obtain $[L_1]_Y= [M_1]_Y$ as required.
\end{proof}

\subsection{Properties of the intersection index which can be deduced
from the surface case} \label{subsec-7}
Motivated by the terminology of line bundles, let us say that a subspace
$L \in {\bf K}_{reg}(X)$ is {\it very ample}
if the Kodaira map $\Phi_L:X \to \p(L^*)$ is an embedding.
Note that $L$ is very ample if and only if for any $k > 0$, $L^k$ is very ample.
The very ample subspaces form a subsemigroup of ${\bf K}_{reg}(X)$.
Let $Y\subset X$ be a smooth algebraic subvariety of a smooth
variety $X$ and let  $L \in {\bf K}_{reg}(X)$ be a
very ample subspace. Then the restriction of $L$ to $Y$
is a very ample space in  ${\bf K}_{reg}(Y)$.

\begin{Th}[A version of the Bertini-Lefschetz theorem]
Let $X$ be a smooth irreducible $n$-dimensional \qp variety
and let $L_1,\dots,L_k\in {\bf K}_{reg}(X)$, $k<n$, be very ample subspaces.
Then there is a Zariski open set ${\bf U}(k)$ in ${\bf
L}(k)=L_1\times\dots\times L_k$ such that for each point  ${\bf
f}(k)= (f_1,\dots,f_k)\in {\bf U}(k)$ the variety defined in $X$ by
the system of equations $f_1=\dots=f_k=0$ is smooth and
irreducible.
\end{Th}

A proof of the Bertini-Lefschetz theorem can be found in \cite[Theorem
8.18]{Hartshorne}

\begin{Th}[A version of Hodge inequality] \label{7.1}
Let $X$ be a smooth irreducible surface and let
$L_1, L_2\in {\bf K}_{reg}(X)$ be very ample subspaces. Then we have
$[L_1,L_2]^2\geq [L_1,L_1][L_2,L_2]$.
\end{Th}

To make the present paper self-contained we give a proof of Theorem
\ref{7.1} using the usual Hodge theory in the appendix. Also as
mentioned in the introduction, the authors have found an elementary
proof of Theorem \ref{7.1} using only the isoperimetric inequality
for planar convex bodies and the Hilbert theorem on degree of a
subvariety in the projective space. This has appeared in a separate
paper (see \cite[Theorem 5.9]{Askold-Kiumars-arXiv} for a preliminary
version, and \cite[Part IV]{Askold-Kiumars-Newton-Okounkov} for the detailed version).

\begin{Th}[Algebraic analogue of the Alexandrov-Fenchel
inequality] \label{7.2} Let $X$ be an irreducible $n$-dimensional
smooth \qp variety and let $L_1,\dots,L_n\in {\bf K}_{reg}(X)$ be
very ample subspaces. Then the following inequality holds
$$[L_1,L_2, L_3,\dots,L_n]^2\geq [L_1,L_1,
L_3,\dots,L_n][L_2,L_2,L_3,\dots,L_n].$$
\end{Th}
\begin{proof}
Consider $n$-tuples $(L_1, L_2,
L_3,\dots,L_n)$, $(L_1,L_1,L_3,\dots,L_n)$ and
$(L_2,L_2,L_3,\dots,L_n)$ of subspaces in the semigroup
${\bf K}_{reg}(X)$. According to the Bertini-Lefschetz theorem and Theorem \ref{3.5}
there is an $(n-2)$-tuple of  functions $f_3\in L_3,\dots, f_n\in
L_n$ such that the system $f_3=\dots =f_n=0$ is non-degenerate and
defines an irreducible surface  $Y\subset X$ for which we have:
$$[L_1,L_2,L_3, \dots,L_n]=[L_1, L_2]_Y,$$ $$[L_1,
L_1,L_3,\dots,L_n]=[L_1, L_1]_Y,$$
$$[L_2,L_2,L_3,\dots,L_n]=[L_2, L_2]_Y.$$  By Theorem \ref{7.1},
$$[L_1,L_2]^2_Y\geq [L_1,L_1]_Y[L_2,L_2]_Y,$$ which proves the theorem.
\end{proof}

The above theorem, in a slightly different form, was proved at the
beginning of $1980$'s by Teissier and the second author. A survey
and a list of references can be found in \cite{Burago-Zalgaller}.

Similar to Section \ref{sec-mixed-volume-properties}, let us
introduce a notation for the repetition of subspaces in the
intersection index. Let $2\leq m\leq n$ be an integer and
$k_1+\dots+k_r=m$ a partition of $m$ with $k_i \in \n$. Consider
the subspaces $L_1, \ldots, L_n \in {\bf K}_{reg}(X)$. Denote by
$[k_1*L_1,\dots, k_r*L_r,L_{m+1},\dots,L_n]$ the intersection index
of these subspaces where $L_1$ is repeated $k_1$ times, $L_2$ is
repeated $k_2$ times, etc. and $L_{m+1},\dots, L_n$ appear once. The
following inequalities follow from Theorem \ref{7.2} exactly in the
same way as the analogous geometric inequalities follow from the
Alexandrov-Fenchel inequality.
\begin{Cor}[Corollaries of the algebraic analogue of
Alexandrov-Fenchel inequality] \label{cor-Alex-Fenchel-regular}
Let $X$ be an $n$-dimensional smooth irreducible \qp variety.

\noindent 1) Let $2\leq m\leq n$ and $k_1+\dots+k_r=m$ with $k_i \in
\Bbb N$. Take very ample subspaces of regular functions $L_1, \ldots, L_n
\in {\bf K}_{reg}(X)$. Then $$[k_1*L_1, \ldots, k_r*L_r, L_{m+1},
\ldots, L_n]^m \geq \prod_{1 \leq j \leq r}[m*L_j, L_{m+1}, \ldots,
L_n]^{k_j}.$$

\noindent 2)(Generalized Brunn-Minkowski inequality) For any fixed
very ample subspaces $L_{m+1}, \ldots, L_n \in {\bf K}_{reg}(X)$, the
function $$F: L \mapsto [m*L, L_{m+1}, \ldots, L_n]^{1/m},$$ is a
concave function on the semigroup of very ample subspaces.
\end{Cor}

\section{Intersection index for subspaces of rational functions}
\label{sec-int-index-rational}
In this section we extend the definitions and results in the
previous section to the following general situation: instead of a
smooth $n$-dimensional \qp variety we take any \qp variety $X$ (possibly singular)
such that all its irreducible components have the same dimension $n$.
And instead of a finite dimensional subspace of regular functions with no base locus,
we take any finite dimensional subspace of rational functions on $X$
such that its restriction to any irreducible component of $X$ is a non-zero subspace.

Let us start with the following lemma dealing with subspaces of
regular functions. Recall that ${\bf K}_{reg}(X)$
denotes the collection of all the finite dimensional subspaces of regular
functions on $X$ with no base locus.
\begin{Lem} \label{lem-Sigma}
Let $X$ be an $n$-dimensional \qp variety (not necessarily smooth)
and let $L_1, \ldots L_n \in {\bf K}_{reg}(X)$. Then for any subvariety
$\Sigma \subset X$ with $\dim(\Sigma) < n$, there is a non-empty  Zariski open subset
${\bf U}_\Sigma \subset {\bf L} = L_1 \times \cdots \times L_n$ such that
for any ${\bf f}=(f_1, \ldots, f_n) \in {\bf U}_\Sigma$ the system
$f_1 = \cdots = f_n = 0$ has no solution in $\Sigma$.
\end{Lem}
\begin{proof}
As in Proposition \ref{3.2} let us fix a basis $g_{ij}$ for each
subspace $L_i$. Consider all the $n$-tuples ${\bf g}_{\bf j} =
(g_{1j_1}, \ldots, g_{nj_n})$ where ${\bf j} = (j_1, \ldots,
j_n)$, which contain one function from each basis $\{g_{ij}\}$. Let
$V_{\bf j}$ be the Zariski open subset in $X$ defined by the
inequalities $g_{ij_i} \neq 0$ , $i=1, \ldots, n$. Since $L_1,
\ldots, L_n \in {\bf K}_{reg}(X)$, the union of the sets $V_{\bf j}$
is $X$. Represent each function $f_i$ in the form $f_i =
\overline{f_i} + c_i g_{ij_i}$, where $\overline{f_i}$ is the
linear combination of all the basis functions $g_{ik}$ except
$g_{i, j_i}$. When $V_{\bf j}$ is non-empty the system can be written
as $$ \frac{\overline{f_1}}{g_{1,j_1}} = -c_1, \cdots ,
\frac{\overline{f_n}}{g_{n,j_n}}=-c_n.$$ In other words in the open set
$V_{\bf j}$ each solution of the system $f_1 = \cdots = f_n =
0$ is a preimage of the point $-{\bf c} = (-c_1, \ldots, -c_n)$ in
$\c^n$ under the map $\phi_{\bf j}: V_{\bf j} \to \c^n$ given by
$$\phi_{\bf j} = (\frac{\overline{f_1}}{g_{1,j_1}}, \cdots ,
\frac{\overline{f_n}}{g_{n,j_n}}).$$ Let $\Sigma_{\bf j} = \Sigma
\cap V_{\bf j}$. The set $\phi(\Sigma_{\bf j})$, has dimension
smaller than $n$ and hence there is a Zariski open subset
${\bf U}_{\bf j} \in {\bf L}$ such that for ${\bf f}=(f_1, \ldots, f_n)
\in {\bf U}_{\bf j}$ the system $f_1 = \cdots = f_n = 0$ has no
solutions in $\Sigma_{\bf j}$. Now take ${\bf U}_\Sigma$ to be the
intersection of all the ${\bf U}_{\bf j}$.
\end{proof}

\begin{Def} \label{def-K-rat} Let $X$ be a \qp variety such that all its
irreducible components have the same dimension $n$. Denote by $\K$ the
collection of all the finite dimensional subspaces $L$ of rational
functions on $X$ such that the restriction of $L$ to any irreducible component of $X$ is a
non-zero subspace.
\end{Def}

A rational function is not necessarily defined every where on $X$.
Nevertheless we can still
define a multiplication on $\K$ exactly as we did for
${\bf K}_{reg}(X)$, and with this multiplication $\K$ becomes a (commutative) semigroup.
We now define a birationally invariant intersection index on the semigroup
$\K$.
\begin{Def}
Let $L \in \K$. We say that a closed subvariety
$\Sigma \subset X$ is said to be {\it admissible} for $(X, L)$ if it satisfies
the following properties:
\begin{enumerate}
\item[(i)] $\Sigma$ contains the singular locus of $X$ and hence $X
\setminus \Sigma$ is smooth.
\item[(ii)] $\dim(\Sigma) < n$.
\item[(iii)] $\Sigma$ contains the supports of all the irreducible divisors on
which a function from $L$ has a pole.
\item[(iv)] $\Sigma$ contains the base locus of $L$, i.e.
the set of points where all the functions in $L$ vanish.
\end{enumerate}

A closed subvariety $\Sigma \subset X$ is said to be {\it admissible} for a
finite collection $L_1, \ldots, L_k \in \K$ if it is
admissible for each $L_i$, $i=1, \ldots, k$.
\end{Def}


\begin{Rem} \label{rem-Kodaira-map-rational}
Take a subspace $L \in \K$ and an admissible subvariety $\Sigma$ for $L$.
The Kodaira map $\Phi_L$ is defined and regular on
$X \setminus \Sigma$. Thus for $L \in \K$,
we can talk about the Kodaira map $\Phi_L$ as a rational map from $X$ to $\p(L^*)$.
\end{Rem}

Let $X$ be a \qp variety such that all its irreducible components have dimension  $n$. Consider $L_1, \ldots, L_n \in
\K$. Take any subvariety $\Sigma$ which is admissible
for $X$ and $L_1, \ldots, L_n$.

\begin{Def}
Let $L_1, \ldots, L_n \in \K$.
The {\it birationally invariant intersection index} (or for short the {\it intersection index})
of the subspaces $L_1, \ldots, L_n$ is the intersection index $[L_1, \ldots,
L_n]_{X \setminus \Sigma}$ of the restrictions of the subspaces
$L_i$ to the smooth variety $X \setminus \Sigma$. We denote it by
the symbol $[L_1, \ldots, L_n]_X$, or simply by $[L_1, \ldots,
L_n]$ when there is no confusion about the variety $X$.
\end{Def}

\begin{Prop}
The birationally invariant index is well-defined.
\end{Prop}
\begin{proof}
The index $[L_1, \ldots, L_n]_{X \setminus \Sigma}$ is defined
because $L_i$ are finite dimensional subspaces of regular function
without base locus on $X \setminus \Sigma$. Let us show that the
index $[L_1, \ldots, L_n]_{X \setminus \Sigma}$ is independent of
the choice of an admissible set $\Sigma$. Let $\Sigma_1, \Sigma_2$
be two admissible subvarieties and put $\Sigma = \Sigma_1 \cup
\Sigma_2$. Then $\Sigma$ is also admissible. Now by Lemma
\ref{lem-Sigma} we have:
$$[L_1, \ldots, L_n]_{X \setminus \Sigma_1} = [L_1, \ldots, L_n]_{X \setminus \Sigma},$$ and
$$[L_1, \ldots, L_n]_{X \setminus \Sigma_2} = [L_1, \ldots, L_n]_{X \setminus \Sigma},$$
which proves the proposition.
\end{proof}


Let $\tau: X \to Y$ be a birational isomorphism between the $n$-dimensional varieties $X$ and $Y$.
\begin{Prop}
For each $n$-tuple $L_1, \ldots, L_n \in {\bf K}_{rat}(Y)$ we have
$$[L_1, \ldots, L_n]_Y = [\tau^*L_1, \ldots, \tau^*L_n]_X.$$
\end{Prop}
\begin{proof}
Since $\tau$ is a birational isomorphism one can find subvarieties
$\Sigma \subset X$ and $\Gamma \subset Y$ such that $\tau$ is an
isomorphism from $X \setminus \Sigma$ to $Y \setminus \Gamma$. By
enlarging $\Sigma$ and $\Gamma$ we can assume that they are
admissible for $\tau^*L_1, \ldots, \tau^*L_n$ on $X$ and $L_1,
\ldots, L_n$ on $Y$ respectively. It is clear that $[L_1, \ldots,
L_n]_{Y \setminus \Gamma} = [\tau^*L_1, \ldots, \tau^*L_n]_{X
\setminus \Sigma}$. Since $$[L_1, \ldots, L_n]_{Y}=[L_1, \ldots,
L_n]_{Y \setminus \Gamma},$$ and $$[\tau^*L_1, \ldots,
\tau^*L_n]_{X} = [\tau^*L_1, \ldots, \tau^*L_n]_{X \setminus
\Sigma},$$ we are done.
\end{proof}

\begin{Prop} \label{prop-1-dim-invariant-index}
Let $L_1 \in \K$ be a one dimensional subspace of
rational functions. Then for any $(n-1)$-tuple of subspaces $L_2,
\ldots, L_n \in \K$ we have
$$[L_1, \ldots, L_n] = 0.$$
\end{Prop}
\begin{proof}
Suppose $L_1$ is spanned by a rational function $f$. An admissible
subvariety $\Sigma$ for $L_1, \ldots,L_n$ contains the hypersurface
$\{f=0\}$. But then no function in $L_1$ vanishes on $X \setminus
\Sigma$ and hence
$$[L_1, \ldots, L_n] = [L_1, \ldots, L_n]_{X \setminus \Sigma} = 0.$$
\end{proof}

\begin{Th} \label{th-obvious-invariant-index}
1) $[L_1,\dots,L_n]$ is a symmetric function of the n-tuples
$L_1,\dots,L_n \in \K$, (i.e. takes the same value
under a permutation of the elements $L_1,\dots,L_n$). 2) The birationally
invariant index is monotone, (i.e. if $L'_1\subseteq
L_1,\dots, L'_n\subseteq L_n$, then $[L_1,\dots,L_n] \geq
[L'_1,\dots,L'_n]$. 3) The birationally invariant index is
non-negative (i.e. $[L_1,\dots,L_n] \geq 0$).
\end{Th}
\begin{proof} Follows from Theorem \ref{3.1}
\end{proof}

\begin{Th}[Multi-linearity] \label{th-multi-lin-invariant-index}
Let $L_1', L_1'', L_2, \ldots, L_n \in \K$
and put $L_1= L_1'L_1''$. Then
$$[L_1,\dots,L_n]=[L'_1,\dots,L_n]+[L''_1,\dots,L_n].$$
\end{Th}
\begin{proof} Follows from Theorem \ref{6.1}
\end{proof}

\begin{Cor} \label{cor-int-index-lin-equ}
Let $L_1, \ldots, L_n \in \K$. Take $1$-dimensional subspaces
$L'_1, \ldots, L'_n \in \K$. Then
$$[L_1, \ldots, L_n] = [L'_1L_1, \ldots, L'_nL_n].$$
\end{Cor}
\begin{proof} Follows from Proposition \ref{prop-1-dim-invariant-index} and
Theorem \ref{th-multi-lin-invariant-index}.
\end{proof}

\begin{Def} \label{def-birational-integral-L}
As before let us say that $f \in \c(X)$ is {\it integral over a
subspace} $L \in \K$ if $f$ satisfies an equation
$$f^m+a_1f^{m-1}+\dots +a_m=0,$$
where  $m>0$ and $a_i \in L^i$, for each $i=1,\ldots, m$.
\end{Def}

\begin{Th}[Addition of integral elements] \label{th-int-element-intersection-index}
Let $L_1 \in {\bf K}_{rat}(X)$ and let
$M_1\in {\bf K}_{rat}(X)$ be the subspace spanned by $L_1$ and
a rational function $g$ integral over $L_1$. Then for any
$(n-1)$-tuple $L_2,\dots,L_n\in {\bf K}_{rat}(X)$ we have
$$[L_1,L_2,\dots,L_n]=[M_1,L_2,\dots,L_n].$$
\end{Th}
\begin{proof} Follows form Theorem \ref{6.2}
\end{proof}

\begin{Def} \label{def-completion}
It is well-known that the collection of all integral elements over $L$ is a vector
subspace containing $L$. Moreover if $L$ is finite dimensional then $\overline{L}$
is also finite dimensional (see \cite[Appendix 4]{Zariski}). It is called
the {\it completion of $L$} and denoted by $\overline{L}$.
\end{Def}

One can also give a description of the completion $\overline{L}$ of
a subspace $L \in \K$ purely in terms of the
semigroup $\K$ (see Theorem \ref{th-completion}).

\begin{Cor}[Intersection index and completion]
\label{cor-invariant-index-int-closure} Let $L_1 \in \K$ and $\overline{L_1}$ its completion as defined above.
Then for any $(n-1)$-tuple $L_2,\dots,L_n\in \K$ we
have
$$[L_1,L_2,\dots,L_n]=[\overline{L_1},L_2,\dots,L_n].$$
\end{Cor}
\begin{proof} Follows form Theorem \ref{th-int-element-intersection-index} and the fact
that $\overline{L_1}$ is finite dimensional and hence can be spanned
by $L_1$ together with a finite number of extra integral elements
over $L_1$.
\end{proof}

Let us say that a subspace $L \in \K$ is {\it very big}
if the Kodaira rational map $\Phi_L$ is an embedding restricted to a
Zariski open set. {\bf We say $L$ is {\it big} if a completion $\overline{L^k}$ is very big for
some $k>0$.}

\begin{Th}[A birationally invariant version of Hodge inequality]
Let $X$ be an irreducible (possibly singular) surface and let
$L_1, L_2\in \K$ be big subspaces. Then we have
$$[L_1,L_2]^2\geq [L_1,L_1][L_2,L_2].$$
\end{Th}
\begin{proof}
{\bf It follows from the multi-linearity of
the intersection index and invariance under the completion (Theorem \ref{th-multi-lin-invariant-index}
and Corollary \ref{cor-invariant-index-int-closure})
that if we replace each $L_i$ with the completion of any power $L_i^{k_i}$
the inequality does not change.} Thus it is enough to prove the theorem for
very big subspaces, in which case it follows from Theorem \ref{7.1}.
\end{proof}

\begin{Th}[Analogue of the Alexandrov-Fenchel
inequality for birationally invariant index] Let $X$ be an
irreducible $n$-dimensional (possibly singular) variety and let
$L_1,\dots,L_n\in \K$ be big subspaces. Then the
following inequality holds:
$$[L_1,L_2, L_3,\dots,L_n]^2\geq [L_1,L_1,
L_3,\dots,L_n][L_2,L_2,L_3,\dots,L_n].$$
\end{Th}
\begin{proof}
As above, {\bf by the multi-linearity and invariance under the completion},
it is enough to prove the theorem for
very big subspaces, and the theorem
in this case follows from Theorem \ref{7.2}.
\end{proof}

The same inequalities as in Corollary \ref{cor-Alex-Fenchel-regular}
hold for the birationally invariant index.

\section{Grothendieck group of subspaces of rational functions and Cartier divisors}
\label{sec-Grothendieck-gp-Cartier}
\subsection{Generalities on semigroups of rational functions} \label{subsec-Grothendieck-group}
Let $K$ be a commutative semigroup (whose operation we denote by
multiplication). $K$ is said to have the {\it cancellation property} if
for $x,y,z \in K$, the equality $xz=yz$ implies $x=y$. Any
commutative semigroup $K$ with the cancellation property can be extended
to an abelian group $G(K)$ consisting of formal quotients $x/y$, $x,
y \in K$. For $x,y,z,w \in K$ we identify the quotients $x/y$ and
$w/z$, if $xz = yw$.

Given a commutative semigroup $K$ (not necessarily with the cancellation
property), we can get a semigroup with the cancellation property by
considering the equivalence classes of a relation $\sim$ on $K$:
for $x, y \in K$ we say $x \sim y$ if there is $z \in K$ with $xz = yz$. The
collection of equivalence classes $K / \sim$ naturally has structure
of a semigroup with cancellation property. Let us denote the group
of formal quotients of $K / \sim$ again by $G(K)$. It is called the {\it
Grothendieck group of the semigroup $K$}. The map which sends $x \in K$ to its
equivalence class $[x] \in K / \sim$ gives a natural homomorphism
$\phi: K \to G(K)$.

The Grothendieck group  $G(K)$ together with the homomorphism $\phi: K \to
G(K)$ satisfies the following universal property: for any other
group $G'$ and a homomorphism $\phi': K \to G'$, there exists a unique
homomorphism $\psi: G(K) \to G'$  such that $\phi' = \psi \circ
\phi$.

\begin{Def}
For two subspaces $L, M \in {\bf K}_{rat}(X)$, we write $L \sim_{rat} M$
if $L$ and $M$ are equivalent as elements of the multiplicative
semigroup ${\bf K}_{rat}(X)$, that is, if there is $N \in {\bf K}_{rat}(X)$
with $LN = MN$.
\end{Def}

From multi-linearity of the intersection index it follows that the
intersection index is invariant under the equivalence of subspaces,
namely if $L_1, \ldots, L_n$ and $M_1, \ldots, M_n \in \K$ are $n$-tuples of subspaces and for each $i$, $L_i
\sim_{rat} M_i$ then
$$[L_1, \ldots, L_n] = [M_1, \ldots, M_n].$$ Hence, generalizing the
situation in Section \ref{sec-B-K-rational}, one can extend the
intersection index to the Grothendieck group of $\K$. We denote this
Grothendieck group by ${\bf G}_{rat}(X)$.

For $L \in \K$ recall that the completion
$\overline{L}$ is the collection of all rational functions integral
over $L$ (Definition \ref{def-completion}). The following result
describes the completion $\overline{L}$ of a subspace $L$ as the
largest subspace equivalent to $L$ (see \cite[Appendix 4]{Zariski}
for a proof).
\begin{Th} \label{th-completion}
For $L \in {\bf K}_{rat}(X)$, the completion $\overline{L}$ is the largest
subspace which is equivalent to $L$. That is, 1) $\overline{L}
\sim_{rat} L$ and 2) if for $M \in {\bf K}_{rat}(X)$ we have $M \sim_{rat}
L$ then $M \subset \overline{L}$.
\end{Th}

Theorem \ref{th-completion} readily implies that the
intersection indices of subspace $L \in \K$ and its
completion $\overline{L}$ are the same (Corollary
\ref{cor-invariant-index-int-closure}).

\begin{Rem}
Let us call a subspace $L$, {\it complete} if $\overline{L} = L$. If
$L$ and $M$ are complete subspaces, then $LM$ is not necessarily
complete. For two complete subspaces $L, M \in \K$,
define $$L * M = \overline{LM}.$$ The collection of complete
subspaces together with $*$ is a semigroup with the cancellation
property. Theorem \ref{th-completion} in fact shows that $L \mapsto
\overline{L}$ gives an isomorphism between the quotient semigroup
$\K / \sim_{rat}$ and the semigroup of complete
subspaces (with $*$).
\end{Rem}

In analogy with the linear equivalence of divisors, we define the linear
equivalence for subspaces of rational functions.
\begin{Def} \label{def-lin-equ-subspace}
Given $L, M \in \K$, we say $L$ is {\it linearly
equivalent to} $M$ if there is a rational function which is not
identically zero on any irreducible component of $X$ and $L = fM$.
We then write $L \sim_{lin} M$. For two classes of subspaces $[L]$,
$[M] \in \K / \sim_{rat}$ we say $[L]$ is {\it
linearly equivalent to} $[M]$ and again write $[L] \sim_{lin} [M]$
if there are $L' \sim_{rat} L$ and $M' \sim_{rat} M$ such that $L'$
is linearly equivalent to $M'$.
\end{Def}
The following is easy to verify. Part 2) is just Corollary
\ref{cor-int-index-lin-equ}.
\begin{Prop} \label{prop-int-index-lin-equ}
1) The linear equivalence $\sim_{lin}$ is an equivalence relation on
$\K$ which respects the semigroup operation.
Similarly, $\sim_{lin}$ is an equivalence relation on the factor
semigroup $\K / \sim_{rat}$ and it respects the
semigroup operation. Hence it extends to an equivalence relation on
the Grothendieck group ${\bf G}_{rat}(X)$. 2) The birationally invariant
index is preserved under the linear equivalence. and thus
induces an intersection index on the factor group ${\bf G}_{rat}(X) / \sim_{lin}$.
\end{Prop}

The intersection index $[L_1,\dots,L_n]_X$ can be computed separately on each irreducible component of the variety $X$. So without lost of generality we can only consider irreducible varieties.

\subsection{Cartier divisor associated to a subspace of rational functions
with a regular Kodaira map} \label{subsec-regular-Kodaira}  A {\it
Cartier divisor} on a projective irreducible variety $X$ is a divisor which can
be represented locally as a divisor of a rational function. Any
rational function $f$ defines a {\it principal Cartier divisor}
denoted by $(f)$. The Cartier divisors are closed under the addition and
form an abelian group which we will denote by $\Div(X)$. A
dominant morphism $\Phi:X\to Y$ between varieties $X$ and $Y$ gives a pull-back
homomorphism $\Phi^*: \Div(Y)\to \Div(X)$. Two Cartier divisors are
linearly equivalent if their difference is a principle divisor. The
group of Cartier divisors modulo linear equivalence is called the
Picard group of $X$ and denoted by $\Pic(X)$. One has an
intersection theory on $\Pic (X)$: for given Cartier divisors
$D_1,\dots, D_n$ on an $n$-dimensional projective variety  there is
an intersection index $[D_1,\dots,D_n]$ which obeys the usual
properties (see \cite{Fulton}).

Now let us return back to the subspaces of rational functions. For a
subspace $L\in {\bf K}_{rat}(X)$, in general, the Kodaira map $\Phi_L$ is
a rational map, possibly not defined everywhere on $X$.

\begin{Def}
We denote the collection of subspaces $L\in {\bf K}_{rat}(X)$
for which the rational
Kodaira map $\Phi_L:X \ratmap \p(L^*)$ extends to a regular map defined everywhere on $X$,
by ${\bf K}_{Cart}(X)$. A subspace $L \in {\bf K}_{Cart}(X)$ is called a {\it subspaces with regular Kodaira map}.
\end{Def}

One can verify that the collection ${\bf K}_{Cart}(X)$ is closed under
the multiplication and under the linear equivalence, i.e. if $L\in
{\bf K}_{Cart}(X)$ and $f$ is a rational function which is not identically
equal to zero on any irreducible component of $X$, then $fL\in
{\bf K}_{Cart}(X)$. Moreover, $\Phi_L= \Phi_{fL}$.

\begin{Def}
To a subspace $L \in {\bf K}_{Cart}(X)$ there naturally corresponds
a Cartier divisor $\mathcal{D}(L)$ as follows: each rational function $h\in L$
defines a hyperplane $H=\{ h=0\}$ in $\p(L^*)$. The divisor $\mathcal{D}(L)$ is
the difference of the pull-back divisor $\Phi ^*_L(H)$ and the
principal divisor $(h)$.
\end{Def}

\begin{Th} \label{th-DL-well-defined}
Let $X$ be an irreducible projective variety. Then:
1) For any $L\in {\bf K}_{Cart}(X)$ the divisor $\mathcal{D}(L)$ is well-defined,
i.e. is independent of the choice of a function $h\in L$.
2) The  map $L\mapsto \mathcal{D}(L)$ is a homomorphism from the semigroup
${\bf K}_{Cart}(X)$ to the semigroup $\Div(X)$ which respects the linear
equivalence.
3) Let ${\bf G}_{Cart}(X)$ denote the Grothendieck group of the semigroup ${\bf K}_{Cart}(X)$.
The map $L \mapsto \mathcal{D}(L)$ extends to a homomorphism
$\rho: {\bf G}_{Cart}(X) \to \Div(X)$.
4) The map $L\mapsto \mathcal{D}(L)$ preserves the intersection index, i.e.
for $ L_1,\dots, L_n\in {\bf K}_{Cart}(X)$ we have $$[L_1,\dots,L_n]=
[\mathcal{D}(L_1),\dots,\mathcal{D}(L_n)],$$ where the right-hand side is the
intersection index of Cartier divisors.
\end{Th}
\begin{proof}
Statements 1) and 2) are obvious. Statement 3) follows from 2). Let us
prove 4). Since both intersection indices for subspaces and for
Cartier divisors are multi-linear, it is enough to prove the
equality  when $L_1=\dots=L_n=L$. The divisor $\mathcal{D}(L)$ is
linearly equivalent to the pull-back of any hyperplane section in
$p(L^*)$. So the self-intersection index of $\mathcal{D}(L)$ is
equal to the intersection index of the pull-back of $n$-generic
hyperplanes, i.e. $ [\mathcal{D}(L),\dots, \mathcal{D}(L)]=
[\Phi^*_L(H_1), \dots, \Phi^*_L(H_n)]$ where the $H_i$ are arbitrary
hyperplanes in $\p(L^*)$. We can choose the hyperplanes $H_i$ in
such a way, that the divisors $\Phi^*_L(H_i)$ intersect transversally
and the intersection points do not belong to the base locus of $L$
as well as the poles of the functions in $L$. In this case, by
the definition, $[L_1,\dots, L_n]$ is equal to the right-hand side in
4) and the proof of theorem is finished.
\end{proof}

\begin{Def}
The subspace $\mathcal{L}(D)$ associated to a Cartier
divisor $D$ is the collection of all rational functions $f$ such that
the divisor $(f)+ D$ is effective. (by definition $0 \in L(D)$.)
\end{Def}

The following well-known fact can be found in \cite[Chap. 2, Theorem
5.19]{Hartshorne}.
\begin{Th} \label{th-LD-finite}
When $X$ is projective
$\mathcal{L}(D)$ is finite dimensional.
\end{Th}

The next proposition is a direct corollary of the definition.
\begin{Prop} Let $L\in {\bf K}_{Cart} (X)$ and put $D = \mathcal{D}(L)$.
Then $L\subset \mathcal{L}(D)$ and $\mathcal{L}(D)\in {\bf K}_{Cart} (X)$.
\end{Prop}

A Cartier divisor $D$ is {\it very ample} if
the Kodaira map of the space $\mathcal{L}(D)$ gives rise to an embedding of
$X$ into the projective space. One verifies that this is equivalent to
$\mathcal{D}(\mathcal{L}(D))=D$. According to the following
well-known theorem, the group of Cartier divisors is generated by
very ample divisors (see \cite[Example 1.2.6]{Lazarsfeld}).

\begin{Th} \label{th-very-ample-div}
Let $X$ be a projective variety. Given a very
ample Cartier divisor $D$ and a Cartier divisor $E$, there is an
integer $N$ such that for any $k>N$ the divisor $E_k=kD+E$ is a very
ample divisor. Thus any Cartier divisor $E$ is the difference of two
very ample divisors, namely, $E=E_k-kD$.
\end{Th}

For a subspace $L\in {\bf K}_{Cart}(X)$ we would like to describe the
subspace $\mathcal{L}(\mathcal{D}(L))$. The answer is based on the following
theorem. It can be found in slightly different forms in
\cite[Chap. 2, Proof of Theorem 5.19]{Hartshorne} and \cite[Appendix 4]{Zariski}.

\begin{Th} \label{th-LDL}
Let $X$ be an irreducible projective variety and let $L\in
{\bf K}_{Cart} (X)$ be such that the Kodaira map $\Phi_L \to
\p(L^*)$ is an embedding. Then:
1) Every element of $\mathcal{L}(\mathcal{D}(L))$ is integral over $L$, i.e.
$\mathcal{L}(\mathcal{D}(L))\subset \overline{L},$
2) Moreover if $X$ is normal then $\mathcal{L}(\mathcal{D}(L))= \overline{L}.$
\end{Th}

Let us use Theorem \ref{th-LDL} to describe
$\mathcal{L}(\mathcal{D}(L))$ in terms of the semigroup
${\bf K}_{Cart}(X)$. For $L,M \in {\bf K}_{Cart}(X)$ we will write $L\sim_{Cart}
M$ if there is $N\in {\bf K}_{Cart}(X)$ with $LN=MN$.

\begin{Cor} \label{cor-LDL}
Let $X$ be an irreducible projective variety and $L \in {\bf K}_{Cart}(X)$.
Also assume that $\Phi _L:X \to \p(L^*)$ is an
embedding. Then:
1) $ \mathcal{L}(\mathcal{D}(L))\sim _{Cart} L$,
2) if $M\sim _{Cart} L$ then $M\subset \mathcal{L}(\mathcal{D}(L)).$
3) finally, let $L_1,L_2\in {\bf K}_{Cart}(X)$ and assume that $\Phi_{L_1}$,
$\Phi_{L_2}$ are embeddings. Then $L_1\sim _{Cart}L_2$ if and only
if $\mathcal{L}(\mathcal{D}(L_1))= \mathcal{L}(\mathcal{D}(L_2))$.
\end{Cor}
\begin{proof}
1) Let $f\in \mathcal{L}(\mathcal{D}(L)) $ By theorem 6.19 we know
that $f$ is integral over $L$. So, there are elements $a_i\in L^i$
such that $f^m+a_1f^{m-1}+\dots+ a_m=0$. Denote by $L(f)$ the vector
space spanned by $L$ and $f$. Denote by $P$ a space
$P=f^{m-1}L+\dots+ L^m$. It is easy to see that: a) the spaces
$L(f)$ and $P$ belong to ${\bf K}_{Cart}(X)$, b) the map $\Phi_{L(f)}:X\to
PF^*$ is am embedding, c) the  inclusion $fP\subset LP$ holds. From
c) we have $LP=L(f)P$. Because $P, L(f)\in {\bf K}_{Cart}(X)$  we showed
that $L\sim _{Cart}L(f)$. Now fix basis $f_1,\dots,f_n$ for
$\mathcal{L}(D)$. Denote by $L_1$ the space $L(f_1)$. As we proved
$L_1\sim _{Cart} L$  and the Kodaira map $\Phi _{L_1}$ is an
embedding. Applying the same arguments to $L_2=L_1(f_2)$  we get
that $L_2\sim_{Cart}L_1\sim _{Cart}L$ and $\Phi_{L_2}$ is an
embedding. Continuing this process we see that $L\sim
_{Cart}\mathcal{L}(\mathcal{D}(L))$. 2) If $M\sim_{Cart}L$ then
$\mathcal{D}(M)=\mathcal{D}(L)$ (see Theorem
\ref{th-DL-well-defined})). Since $M\subset L (\mathcal{D}(M))$ we
have that $M \subset \mathcal{L}(\mathcal{D}(L))$. Part 3) of the
corollary follows from 1) and 2).
\end{proof}

The last statement in the corollary shows that the spaces $L$ for
which the Kodaira map is an embedding behave especially good
with respect to the map $L\to \mathcal{D}(L)$. One easily
verifies the following.

\begin{Lem} \label{lem-Kodaira-embed}
Let $P\in {\bf K}_{Cart}(X)$ be such that the Kodaira
map $\Phi_P$ is an embedding. 1) if $M\in {\bf K}_{Cart}(X)$ and
$P\subset M$ or 2) if $L\in {\bf K}_{Cart}(X)$ and $M=PL$, then
$\Phi _M$ is also an embedding.
\end{Lem}

Let $X$ be a projective variety. Consider the group homomorphism
$\rho$ from the Grothendieck group ${\bf G}_{Cart}(X))$ to the group $\Div(X)$ (see Theorem \ref{th-DL-well-defined}).

\begin{Th} \label{th-rho-iso}
For an irreducible projective variety $X$ the homomorphism
$$\rho: {\bf G}_{Cart}(X) \to \Div(X),$$ is an isomorphism which
preserves the intersection index.
\end{Th}

\begin{proof}
By Theorem \ref{th-DL-well-defined} we know that $\rho$ is a
homomorphism which preserves the intersection index.  The group
$\Div(X)$ is generated by very ample divisors which belong to the
image of ${\bf K}_{Cart}(X)$ under the map $\rho$ (see Theorem
\ref{th-very-ample-div}). So the homomorphism $\rho $ is onto. Take
two spaces $L_1$ and $L_2$ from ${\bf K}_{Cart}(X)$ for which the Kodaira maps
$\Phi_{L_1}$, $\Phi_{L_2}$ are embeddings. According to the last
statement in Corollary \ref{cor-LDL} the map $\rho (L_1)=\rho(L_2)$
if and only if the subspaces $L_1$ and $L_2$ define the same element in the
Grothendieck group ${\bf G}_{Cart}(X)$. From this and
Lemma \ref{lem-Kodaira-embed} one can see that $\rho$ has no kernel.
\end{proof}

Recall that for two subspaces $L$, $M$ we write $L \sim_{lin} M$ if there
is a rational function $f$ which is not identically equal to zero on
any irreducible component of $X$ with $L=fM$.

\begin{Cor}
For an irreducible projective variety $X$ the isomorphism
$\rho$  induces the isomorphism $$\tilde{\rho}: {\bf G}_{Cart}(X)/\sim_{lin}\to \Pic(X),$$
which preserves the intersection index.
\end{Cor}

\begin{Rem}
Theorem \ref{th-rho-iso} shows that for a projective variety $X$ the
Grothendieck group of the semigroup ${\bf K}_{Cart}(X)$ can naturally be
identified with the group $\Div(X)$ of Cartier divisors. From this
one can show that the Grothendieck group ${\bf G}_{rat}(X)$, of the semigroup $\K$, can naturally be
identified with the direct limit of the  groups $\Div(Y)$ of Cartier
divisors on projective birational models $Y$ of $X$.
\end{Rem}

\begin{Rem}
If $X$ is not only projective  but also normal then
using Theorem \ref{th-LDL} one proves that: a) ${\bf K}_{Cart}(X)$ is closed
under completion, i.e. if $L\in {\bf K}_{Cart}(X)$ then $\overline{L} \in
{\bf K}_{Cart}(X)$ and moreover $L\sim_{Cart} \overline{L}$. b) Let $L,M\in
{\bf K}_{Cart}(X)$. Then $L\sim _{rat} M$ if and only if $L\sim _{Cart}
M$.
\end{Rem}

\section{Topological and algebro-geometric proofs of properties of intersection index}
\label{sec-appendix} In this section we provide alternative proofs for the
main properties of the birationally invariant intersection index. The proofs are essentially based
on differential geometry techniques. These techniques
fit into the context of differential geometry in a specially nice
way when the variety $X$ is smooth and compact (or more generally if
it is normal and complete) and when the subspaces of rational
functions $L_i \in \K$, appearing in the intersection
index, give rise to regular Kodaira maps $\Phi_{L_i}: X \to
\p(L_i^*)$. But as we will see the topological arguments are also
applicable in the general case. Note that a similar
construction allows one to prove the properties of the
intersection index (for subspaces of rational functions) using
standard intersection theory for Cartier divisors on a complete
variety.

We also provide a proof of the Hodge inequality for
the birationally invariant index, which
relies on Hodge theory and other well-known methods of algebraic
geometry. In particular, the proof of algebraic analogue of
the Alexandrov-Fenchel inequality mostly follows its proof in
\cite{Burago-Zalgaller}. (As mentioned in the introduction we
recently found an elementary and unusual proof of this inequality;
see \cite{Askold-Kiumars-arXiv} for a preliminary version.)

We first recall the necessary statements from complex geometry.

All smooth complex varieties have standard orientation coming
from the complex structure. Let $Y$ be a complete $n$-dimensional
algebraic subvariety in a smooth projective algebraic variety $M$.
Let $Y^{sm}$ be the locus of smooth points of the variety Y.
It is possible to apply the results in differential geometry to singular
algebraic varieties because of the following: the variety $Y^{sm}$
defines a $2n$-dimensional cycle in the manifold $M$.

\begin{Th} [see \cite{Griffiths-Harris}]
For every smooth $2n$-form $\tau$ on $M$ its integral along the manifold
$Y^{sm}$ with its complex orientation is well defined. Moreover,
for exact forms $\tau$ this integral vanishes.
\end{Th}

The complex projective space has a standard Kaehler $2$-form. We
recall its definition: take a positive definite hermitian form $H$
on a complex vector space $L$. Then we can define a form
$\tilde{\omega}_H$ on $L\setminus\{0\}$ by the formula
$$\tilde{\omega}_H = \frac{i}{2\pi}\partial\bar{\partial}\log{H}.$$
Let $\pi:L\setminus\{0\}\to \p(L)$ be the canonical
projection. The following is well-known,

\begin{Th}[see \cite{Griffiths-Harris}]
There exists a unique form $\omega_H$ on
$\p(L)$ such that $\tilde{\omega}_H=\pi^*\omega_H$. This form $\omega_H$
is closed and its cohomology class generates the cohomology ring of $\p(L)$.
The integral of $\omega_H$ on any projective line $\p^1\subset\p(L)$ is equal
to $1$. The form $\omega_H$ is the Poincare dual to any hyperplane in $\p(L)$.
\end{Th}

The form $\omega_H$ is called the {\it Kaehler form on $\p(L)$
associated to the Hermitian form $H$ on $L$}. (The Hermitian form $H$
on $L$ also defines the Fubini-Study metric on $\p(L)$ and the form
$\omega_H$ can be recovered uniquely from the Fubini-Study metric.)
For a one-dimensional space $L$ the form $\omega_H$ is equal to
zero as $\p(L)$ consists of a single point.

We now go back to our case of interest. Let $X$ be an irreducible $n$-dimensional
\qp variety. {\it A model of  the variety $X$} is a
complete algebraic variety, which is birationally isomorphic to $X$.
For the topological description of the intersection indices of
$n$-tuples of subspaces in the semigroup $\K$ we will
consider models of the variety $X$. For different $n$-tuples of
elements of the semigroup $\K$ different models may
be needed. In particular, when $X$ is normal and projective, and for all
the subspaces $L \in {\bf K}_{Cart}(X)$ with regular Kodaira map, it
suffices to take $X$ itself as a model.

In general, there exists a common model for all the $n$-tuples of
subspaces chosen from a fixed finite set. We now describe the
construction of this model. Fix a finite sequence $L_1,\ldots,L_k$
of subspaces. For convenience  we allow repetition of subspaces in
the sequence.

Let $\Sigma$ be an admissible set for $X$ and the sequence
$L_1,\ldots, L_k$ of subspaces from the semigroup $\K$. Let $\Phi_{L_i} : X \setminus \Sigma \to
\p(L_i)^*$ be the Kodaira map associated to $L_i$. Let the
\qp variety $X$ be embedded in $\p^N$. Denote the
product $\p(L_1^*)\ldots\times\p(L_k^*)$ by $\p(\bf{L}^*)$. Consider
the mapping $\Phi_{\bf L}:X \setminus \Sigma \to \p(\bf{L}^*)$
which is the product of the Kodaira maps $\Phi_{L_i}$. Let
$\Gamma\subset\p^N\times\p(\bf{L}^*)$ be the graph of this mapping,
i.e. $\Gamma$ is the set of points $(x,y_1,\ldots,y_k)$ where
$x\in X \setminus \Sigma$ and $y_i=\Phi_{L_i}(x)$, $i=1, \ldots, k$. Define
$\overline{\Gamma}$ to be the closure of the graph $\Gamma$ in
$\p^N\times\p(\bf{L}^*)$. Let $\rho:\overline{\Gamma}\to\p^{N}$ be
the restriction of the projection $\p^N\times\p(\bf{L}^*) \to
\p^N$ to the complete subvariety $\overline{\Gamma}$. The mapping
$\rho$ is a birational isomorphism between the complete variety
$\overline{\Gamma}$ and the variety $X$.

\begin{Def}
The complete variety
$\overline{\Gamma}\subset\p^N\times\p(\bf{L}^*)$ is called {\it the
model of the variety $X$ associated to the sequence of spaces
$L_1,\ldots,L_k$}.
\end{Def}

\begin{Lem}
If $X$ is normal and complete and $L_1, \ldots, L_k \in {\bf
K}_{Cart}(X)$ are subspaces with regular Kodaira maps, then the model of
$X$ corresponding to $L_1, \ldots, L_k$ is isomorphic to $X$.
\end{Lem}
\begin{proof}
Under the assumptions of the lemma the graph $\Gamma$ is defined
over the whole $X$. Since $X$ is complete, $\Gamma$ is
closed and thus $\overline{\Gamma} = \Gamma$ is isomorphic to $X$.
\end{proof}

\begin{Rem}
Using intersection theory of Cartier divisors on a model
$\overline{\Gamma}$ of a variety $X$ corresponding to a sequence
$L_1, \ldots, L_k$ of subspaces of rational functions, one can prove
all the properties of the intersection index discussed before, for the $n$-tuples
$L_{i_1}, \ldots, L_{i_n}$ taken from the sequence $L_1, \ldots,
L_k$.
\end{Rem}

In the variety $\overline{\Gamma}$ there is a Zariski open subset
$W=\rho^{-1}(X\setminus\Sigma)$ isomorphic to
$X\setminus\Sigma$. If we identify $W$ with
$X\setminus\Sigma$ using the map $\rho$, then the
restriction of the projection $\pi_i: \p^N \times \p({\bf L}^*) \to \p(L_i^*)$ to $W$ is the Kodaira
map $\Phi_{L_i}$.

Choose an arbitrary collection of $n$ subspaces from the sequence
$L_1,\ldots,L_k$. Without loss of generality, assume that it is
$L_1,\ldots,L_n$.

Below we will give an integral formula for the intersection index
(see Corollary \ref{cor-int-index-integral}).

Recall basic facts from the topology of the product of projective
spaces $M = \p^N\times\p(\bf{L}^*)$. Since the Kaehler form
$\omega_{H_i}$ on $\p(L_i^*)$ is the Poincare dual to the
hypersurface $P_i\subset\p(L_i^*)$, the form $\pi_i^*\omega_{H_i}$
on $M$ is Poincare dual to the hypersurface $\pi^{-1}P_i$, and the
intersection $Z=\pi_1^{-1}P_1\cap\ldots\cap\pi_n^{-1}P_n$ is
the Poincare dual to the form $\pi_1^*\omega_{H_1}\wedge\ldots\wedge
\pi_n^*\omega_{H_n}$. It means that for an arbitrarily small tubular
neighborhood $U$ of the manifold $Z$ equipped with a smooth
projection $g:U\to Z$ there is a closed $2n$-form $\Omega_U$ on the
manifold $M$ having the following properties: 1) the support of the
form $\Omega_U$ is contained in the tubular neighborhood $U_Z$. 2)
The form $\Omega_U$ is cohomologous to the form
$\pi_1^*\omega_{H_1}\wedge\ldots\wedge \pi_n^*\omega_{H_n}$. 3) Let
$g:U\to Z$ be the smooth projection in the definition of the
tubular neighborhood $U$, and let $F_a=g^{-1}a$ be the fiber of this
projection over an arbitrary point $a\in Z$, oriented
compatible with the natural coorientation of the submanifold $Z$ in
$M$. Then $\int_{F_a}\Omega_U=1$.

\begin{Th} \label{th-integral-int-index}
Let $L_1,\ldots,L_n$ be an $n$-tuple of subspaces taken from a sequence
of $k \geq n$ subspaces from $\K$. Let $\overline{\Gamma}$ be a model of the variety
$X$ associated to this sequence. Suppose that the cycle
$Z=\pi_1^{-1}P_1\cap\ldots\cap\pi_n^{-1}P_n$ does not intersect
the singular set $O$ of the variety $\overline{\Gamma}$, and intersects
its non-singular part $Y=\overline{\Gamma}\setminus O$
transversally. Then the number $\#(Z \cap Y)$ of intersection points of
the cycle $Z$ with the manifold $Y$ is
$$\int_Y{\pi_1^*\omega_{H_1}\wedge\ldots\wedge
\pi_n^*\omega_{H_n}}$$.
\end{Th}
\begin{proof}
The variety $Z$ does not intersect the closed set $O$, and hence the set
$Z\cap Y$ does not contain the limit points of $O$. It follows from the
transversality condition that the set $Z\cap Y$ is finite. Choose a
sufficiently small tubular neighborhood $U$ of the cycle $Z$ so that for
every point $a\in Z\cap Y$ the connected component $Y_a$ of the
intersection of $Y$ with the closure of the tubular neighborhood
$\overline{U}$ defines the same relative cycle in the homology group
of the pair $(\overline{U},\partial\overline{U})$ as the fiber
$F_a$. Then $$\int_Y{\pi_1^*\omega_{H_1}\wedge\ldots\wedge
\pi_n^*\omega_{H_n}}=\int_Y{\Omega_U}=\sum_{a\in Z\cap
Y}{\int_{Y_a}{\Omega_U}}= \#(Z \cap Y),$$ as required.
\end{proof}

This theorem not only gives an alternative proof of the well-definedness of
the definition of the intersection index $[L_1,\ldots,L_n]$, but also gives an
explicit formula for it:
\begin{Cor} \label{cor-int-index-integral}
The birationally invariant index $[L_1,\ldots,L_n]$ is
well-defined and is equal to
\begin{eqnarray*}
[L_1,\ldots,L_n] &=& \int_Y{\pi_1^*\omega_{H_1}\wedge\ldots\wedge
\pi_n^*\omega_{H_n}}, \cr
&=& \int_{X\setminus\Sigma}{(\pi_1\circ\rho^{-1})^*
\omega_{H_1}\wedge\ldots\wedge (\pi_n\circ\rho^{-1})^*\omega_{H_n}}. \cr
\end{eqnarray*}
\end{Cor}

\begin{proof}
Every function $f_i\in L_i$ defines a linear functional on the
vector space $L_i^*$, and moreover, a nonzero $f_i$ defines a hyperplane
$L_{f_i}$ in $L_i^*$ on which this functional vanishes. Let
$P_{f_i} \subset \p(L_i^*)$ be the projectivization of $L_{f_i}$. By
the Bertini theorem there is a Zariski open subset $U\subset
L_1\times\ldots\times L_n$ consisting of
${\bf f}=(f_1,\ldots,f_n)$ such that the hyperplanes $P_{f_i} \subset \p(L_i^*)$ have the
following properties: 1) the cycle
$Z=\pi_1^{-1}P_{f_1}\cap\ldots\cap\pi_n^{-1}P_{f_n}$ satisfies the
conditions of Theorem \ref{th-integral-int-index}, 2) all
the intersection points of the cycle $Z$ with the variety
$\overline{\Gamma}$ belong to the set
$\rho^{-1}(X\setminus\Sigma)\subset\overline{\Gamma}$.
According to the theorem, for $\bf{f}\in U$ the system
$f_1=\ldots=f_n=0$ in $X\setminus\Sigma$ has only
non-degenerate roots. Their number does not depend on the choice of
the set of equations $\bf{f}\in U$ and is equal to
$\int_Y{\pi_1^*\omega_{H_1}\wedge\ldots\wedge \pi_n^*\omega_{H_n}}$.
The latter integral is equal to
$\int_{\rho^{-1}(X\setminus\Sigma)}{\pi_1^*\omega_{H_1}\wedge\ldots\wedge
\pi_n^*\omega_{H_n}}$ (and hence the last integral is well-defined).
This proves the corollary.
\end{proof}

The above formula for the index provides a different proof of its
multi-linearity. Let $L_1,L_2$ be two spaces of regular functions on
some algebraic variety and let $L_1 L_2$ be their product. The space
$L_1 L_2$ is a factor space of the tensor product $L_1\otimes L_2$
and hence the space $(L_1 L_2)^*$ is a subspace of $(L_1\otimes
L_2)^*$. Denote by $\tau:(L_1 L_2)^*\to (L_1\otimes L_2)^*$ the
corresponding embedding. Let $H_1,H_2$ be positive-definite
Hermitian forms on the spaces $L_1^*, L_2^*$ and $H_1\otimes H_2$ be
the positive-definite form on $(L_1\otimes L_2)^*$ whose value on
the product $l_1\otimes l_2$ of two vectors $l_1,l_2$ is
$H_1(l_1)H_2(l_2)$.

\begin{Def}
We call the form $H_1 H_2$ on the space $(L_1 L_2)^*$, induced by the form
$H_1\otimes H_2$ and the embedding $\tau:(L_1 L_2)^*\to  (L_1\otimes
L_2)^*$, the {\it product of the forms $H_1$ and $H_2$}.
\end{Def}

\begin{Lem}
Let $L_1,L_2\in \K$. Let $H_1$, $H_2$ be positive-definite Hermitian
forms on the spaces $L_1^*$, $L_2^*$ respectively, and let $H_1 H_2$ be their product
defined on $(L_1 L_2)^*$. Then
$$\tilde{K}_{L_1}^*\omega_{H_1}+\tilde{K}_{L_2}^*\omega_{H_2}=\tilde{K}_{L_1
L_2}^*\omega_{H_1 H_2}.$$
\end{Lem}
\begin{proof}
Since $\tau\circ \Phi_{L_1 L_2}=\Phi_{L_1}\otimes \Phi_{L_2}$, for
every $x\in X$ we have $H_1 H_2(\Phi_{L_1
L_2}(x))=H_1(\Phi_{L_1}(x))H_2(\Phi_{L_2}(x))$. It follows that
$$\frac{i}{2\pi}\partial\bar{\partial}\log{H_1 H_2(\Phi_{L_1 L_2}(x))}=
\frac{i}{2\pi}\partial\bar{\partial}\log{H_1(\Phi_{L_1}(x))}+\frac{i}{2\pi}
\partial\bar{\partial}\log{H_2(\Phi_{L_2}(x))}.$$
\end{proof}

\begin{Cor}
The index $[L_1,\ldots,L_n]$ is multi-linear. That is, if $L_1=L'_1 L''_2$, then
$$[L_1,L_2,\ldots,L_n]=[L'_1,L_2,\ldots,L_n]+[L''_1,L_2,\ldots,L_n].$$
\end{Cor}
\begin{proof}
Build a model $\overline{\Gamma}$ of the variety $X$ associated to
the sequence of subspaces $L'_1$, $L''_1$, $L_1, \ldots, L_n$. Let
$H'_1$, $H''_1$, $H_2, \ldots, H_n$ be Hermitian forms on the spaces
${L'_1}^*, {L''_1}^*, L_2^*,\ldots, L_n^*$ respectively.
Also let $\pi'_1$, $\pi''_1$ denote the projections
corresponding to the subspaces $L'_1$, $L''_1$ respectively.
Let $H_1=H'_1H''_1$ be the product form on $L_1^*$.
Denote by $Y$ the set of non-singular points on the
variety $\overline{\Gamma}$. It follows from the lemma that the
restrictions of the forms ${\pi_1}^*\omega_{H_1}$ and
${\pi'_1}^*\omega_{H'_1}+ {\pi''_1}^*\omega_{H''_1}$ to the
manifold $Y$ coincide. Hence the restrictions to $Y$ of the products
of these forms by the form
$\pi_2^*\omega_{H_2}\wedge\ldots\wedge\pi_n^*\omega_{H_n}$ also
coincide. The claim now follows from the integral formula for the
index.
\end{proof}

We will need the following corollary from Hodge theory. Let $Y$ be a
smooth projective surface. Consider the symmetric form
$B(\omega_1,\omega_2)=\int_Y{\omega_1\wedge\omega_2},$
defined on the space of real-valued smooth $(1,1)$-forms. Also,
in the space $H^{1,1}(Y)$, consider the non-negative cone $C$ consisting of
the forms $\omega$ whose value at each point of the surface $Y$ on each
pair of the vectors $(v,iv)$ is non-negative.

\begin{Cor}[Corollary of the Hodge index theorem] \label{cor-of-Hodge-index}
For any two forms $\omega_1,\omega_2\in C$ the following inequality holds:
\begin{equation} \label{eqn-Hodge-inequ}
B^2(\omega_1,\omega_2)\geq B(\omega_1,\omega_1) B(\omega_2,\omega_2).
\end{equation}
\end{Cor}
\begin{proof}
According to the Hodge index theorem, the quadratic form $B(\omega,\omega)$,
defined on the $(1,1)$-component of the cohomology of the
surface $Y$, is positive-definite on a one-dimensional subspace and
negative-definite on its orthogonal compliment. The positive closed
$(1,1)$-forms lie in a connected component $B(\omega,\omega) \geq
0$. Now (as in the proof of the classical Cauchy-Schwartz
inequality) consider the line $\ell = \{t\omega_1 + \omega_2 \mid t
\in \r \}$ passing through $\omega_2$ and parallel to $\omega_1$.
The line $\ell$ does not completely lie outside the positive cone
$C$ and hence the quadratic polynomial $Q(t) = B(t\omega_1+\omega_2,
t\omega_1+\omega_2)$ attains both positive and negative (or possibly
zero) values. It follows that the discriminant of $Q$ is
non-negative. But discriminant of $Q$ is equal to $4(B(\omega_1,
\omega_2)^2- B(\omega_1,\omega_1)B(\omega_2,\omega_2))$ which proves
the inequality (\ref{eqn-Hodge-inequ}).
\end{proof}

Finally, we recall the following classical theorem about the resolution
of singularities for surfaces:

\begin{Th}[Resolution of singularities for surfaces] \label{th-res-sing-surface}
For every complete projective surface $\overline{\Gamma}$ there
exists a complete nonsingular projective surface $Y$ and a morphism
$g:Y \to \overline{\Gamma}$ which is a birational
isomorphism between $Y$ and $\overline{\Gamma}$.
\end{Th}

\begin{Th} [A version of Hodge inequality] \label{th-Hodege-appendix}
Let $X$ be a quasi-projective irreducible surface and $L_1,L_2\in
\K$. Then the following inequality holds: $$[L_1,L_2]^2
\geq [L_1,L_1][L_2,L_2].$$
\end{Th}
\begin{proof}
Let $\overline{\Gamma}$ be a complete model for the surface $X$
associated to the subspaces $L_1,L_2$. Consider the nonsingular
projective surface $Y$ with the morphism $g:Y\to
\overline{\Gamma}$ from Theorem \ref{th-res-sing-surface}.
Consider the subspaces $F=(\rho\circ g)^*L_1$ and
$G=(\rho\circ g)^*L_2$ of rational functions on the surface $Y$.
We have Kodaira maps $\Phi_F$ and $\Phi_G$ corresponding to these subspaces.
Fix Kaehler forms $\omega_{H_1}$, $\omega_{H_2}$ on the spaces
$\p(F^*)$, $\p(G^*)$, then we can associate
non-negative smooth closed $(1,1)$-forms $\omega_1=(\pi_1\circ
g)^*\omega_{H_1}$ and $\omega_2=(\pi_2\circ g)^*\omega_{H_2}$ (on the
surface $Y$) to the subspaces $F$ and $G$ respectively. From the integral formula for the index we have
$[F,F]=\int_Y{\omega_1\wedge\omega_1}$,$[F,G]=\int_Y{\omega_1\wedge\omega_2}$
and $[G,G]=\int_Y{\omega_2\wedge\omega_2}$. Corollary \ref{cor-of-Hodge-index}
now implies that $[F,G]^2\geq[F,F][G,G]$. But
$[F,G]=[L_1,L_2]$,$[F,F]=[L_1,L_1]$,$[G,G]=[L_2,L_2]$,
so the theorem is proved.
\end{proof}

From the version of Hodge inequality proved above we get an
analogue of the Alexandrov-Fenchel inequality for the intersection index (see
\cite{Burago-Zalgaller}).

\end{document}